\documentclass[a4paper,10pt]{amsart}
\usepackage[utf8]{inputenc}
\usepackage[english]{babel}
\usepackage{verbatim,color}
\usepackage{enumitem}
\usepackage{ulem}
\usepackage{amsmath,latexsym,amssymb,verbatim}
\usepackage{xcolor}

\newtheorem{lema}{Lemma}[section]
\newtheorem{theo}[lema]{Theorem}

\newtheorem{prop}[lema]{Proposition}
\newtheorem{coro}[lema]{Corollary}

\newtheorem{defi}[lema]{Definition}
\newtheorem{rema}[lema]{Remark}

\newtheorem{exam}[lema]{Example}
\newtheorem{problem}[lema]{Problem}

\newcounter{teoremaganso}

\newtheorem {bigtheo} [teoremaganso] {Theorem}

\newcommand{\R}{\mathbb{R}}
\newcommand{\C}{\mathbb{C}}
\newcommand{\CP}{\mathbb{CP}}
\newcommand{\Q}{\mathbb{Q}}
\newcommand{\N}{\mathbb{N}}
\newcommand{\Z}{\mathbb{Z}}

\newcommand{\Aff}{\text{Aff}}
\newcommand{\D}{\text{D}}
\title[Infinitesimal and tangential 16-th Hilbert problem on zero-cycles]{
Infinitesimal and tangential \\ 16-th Hilbert problem \\ on zero-cycles
}

\author{J.L. Bravo, P. Marde\v si\'c, D. Novikov and J. Pontigo-Herrera}

\subjclass[2010]{34C08, 37G15, 14N10, 14H50}

\keywords{Infinitesimal and tangential 16-th Hilbert problem, abelian integrals, Chebyshev property, deformation of integrable systems}

\thanks{
The first, second and fourth authors were partially supported by grant number PID2020-118726GB-I00 funded
by MCIN/AEI/10. 13039/501100011033 and “ERDF A way of making Europe”.
The second, third and fourth author were partially supported by the Croatian Science Foundation grant PZS-2019-02-3055 and by the Mexican Papiit (Dgapa UNAM) IN103123 project. 
The second author was also supported by the Croatian Science Foundation grant IP-2022-10-9820, the bilateral Hubert-Curien Cogito grants:  \emph{Fractal and transserial approach to differential equations}(2021-22) and  \emph{Singularities of analytic dynamical systems and applications} (2023-2024),
 Dgapa UNAM and Intercambio Acad\'emico, CIC, UNAM,
IRL 2001 LaSol CNRS-UNAM, 
EIPHI Graduate School (contract ANR-17-EURE-0002).
Third author was also supported by the Israel Science Foundation (grant No. 1347/23) and by funding received from the MINERVA Stiftung with the funds from the BMBF of the Federal Republic of Germany. }

\address{J.L. Bravo, Departamento de Matem\'{a}ticas,
Universidad de Extremadura, 06006 Badajoz, Spain}
\email{trinidad@unex.es}

\address{P. Marde\v si\'c, Universit\'e de Bourgogne, Institut de Mathématiques de Bourgogne, UMR 5584 du CNRS, 9 avenue Alain Savary, 21078 Dijon, France.
and University of Zagreb, Faculty of Science, Department of Mathematics, Bijenička cesta 30, 10000 Zagreb, Croatia.
}
\email{pavao.mardesic@u-bourgogne.fr}

\address{D. Novikov, Department of Mathematics, 
Weizmann institute of Science
Rehovot 7610001 Israel}
\email{dmitry.novikov@weizmann.ac.il}

\address{J.D. Pontigo-Herrera, Instituto de Matem\'aticas, Universidad Nacional Aut\'onoma de 
	M\'exico 	(UNAM), 	\'Area de la Investigaci\'on Cient\'ifica, Circuito 
	exterior, Ciudad 	Universitaria, 04510, Ciudad de M\'exico, M\'exico}
\email{pontigo@matem.unam.mx}

\begin{document}

\begin{abstract}
In this paper,
given two polynomials $f$ and $g$ of one variable and a $0$-cycle $C$ of  $f$, 
we consider the deformation $f+\epsilon g$.
We define two functions: the \emph{displacement function}
$\Delta(t,\epsilon)$ and its first order approximation: the \emph{abelian integral} $M_1(t)$.

The \emph{infinitesimal} and \emph{tangential 16-th Hilbert problem} for zero-cycles are problems of counting isolated regular zeros of  $\Delta(t,\epsilon)$, for $\epsilon$ small, or of $M_1(t)$, respectively.

We show that the two problems are not equivalent and find optimal bounds, in function of the degrees of $f$ and $g$, for the infinitesimal and tangential 16-th Hilbert problem on zero-cycles. These two problems are the zero-dimensional analogue of the classical infinitesimal and tangential 16-th Hilbert problems for vector fields in the plane. 
\end{abstract}


\maketitle

\section{Introduction and Motivation}

This article is dedicated to the solution of the \emph{zero-dimensional} version of the \emph{infinitesimal} and \emph{tangential 16-th Hilbert problem} (shorter just infinitesimal and tangential problems). 
The problems  are inspired by the  classical infinitesimal and tangential 16-th Hilbert problems for deformations of integrable systems in the plane. The classical infinitesimal 16-th Hilbert problem asks for the number of (one-dimensional) limit cycles (i.e. isolated cycles) being born by deformation from an integrable system.  The tangential problem (in our terminology) is the first order version of the infinitesimal 16-th problem. 

The classical infinitesimal and tangential 16-th Hilbert problem are far from being solved. 
The only general result is a very rough bound for the number of solutions of the tangential problem \cite{BNY}  and no general result is known for the infinitesimal problem.
The zero-dimensional versions of these problems, studied here, boil down to purely algebraic problems and can be solved. Nevertheless, the two zero-dimensional problems are surprisingly rich, with many remaining open questions given in the last section. 

Let us note first that the two problems (infinitesimal and tangential) are not equivalent. There exist \emph{alien limit} cycles (solutions of the infinitesimal problem) not corresponding to solutions of the tangential problem.

It has been conjectured by Arnol'd that in the classical tangential problem the abelian integrals corresponding to \emph{natural} problems form a Chebychev system (i.e. the number of zeros of  these functions is strictly less than the dimension of the space of functions). 
We show that it is far from being the case for either the tangential or the infinitesimal problem on $0$-cycles. 

This work is a kind of continuation of \cite{ABCM}, where the infinitesimal and tangential versions of the \emph{center problem} on zero-cycles were solved. 

We recall also the work \cite{GM} of Gavrilov and Movasati, who, in our terminology, studied the \emph{tangential} problem on \emph{simple} $0$-cycles. 
They call it the infinitesimal problem, but we prefer to keep the terminology infinitesimal for the full deformation and call the first-order deformation  problem tangential.
Our generalization with respect to \cite{GM} is hence two-fold: we study \emph{any type} of $0$-cycles instead of only simple cycle and study the \emph{infinitesimal} problem which was not addressed in \cite{GM}. We also   show that the bounds we obtain are optimal and determine when \emph{alien limit cycles} can exist. 

Our two problems on $0$-cycles themselves can be formulated in an elementary way without any reference to the classical problems on $1$-cycles. This is done formally in the next subsection and the motivation is developed further in the subsequent subsection.

\subsection{Infinitesimal and Tangential 16-th Hilbert problem on zero-cycles}

Given a non-constant polynomial function $f \in \C[z]$, recall that $z_0\in\C$ is a {\it critical point}
of $f$ if $f'(z_0)=0$, and its associated {\it critical value} is $t_0\in\mathbb{C}$ 
such that $f(z_0)=t_0$. If $t$ is not a critical value, we say that $t$ is {\it regular}.

We denote by $\Sigma$ the set of all \emph{critical values} of $f$, which is a finite set. 
Let $m>1$ be the degree of $f$. Then, for \emph{regular values} $t \in \C \setminus
\Sigma$, the set $f^{-1}(t)$ consists of $m$ different points: $z_j(t)$, $j=1,\ldots,m$. By the implicit function theorem, one can push locally each
solution $z_j(t)$ to nearby values of $t$, thus defining \emph{multi-valued algebraic functions} $z_i(t)$, $t \in \C \setminus \Sigma$.

Let $(z_j(t))_{1\leq j\leq m}$ denote an $m$-tuple of (distinct) analytic preimages $z_j(t)\in f^{-1}(t)$,
where $t \in \C \setminus \Sigma$. 
We define a {\it zero-dimensional chain} (shorter {\it chain}) of $f$ as the divisor i.e. formal sum of the form
\begin{equation*}
C(t)=\sum_{j=1}^m n_j z_j(t),\quad n_j\in \mathbb{Z}. 
\end{equation*}
We say that a chain is a  {\it zero-dimensional cycle} (shorter {\it cycle}) if 
\begin{equation}\label{cycle}
\sum_{j=1}^m n_j=0.
\end{equation}

A cycle of the form 
\begin{equation}\label{simple}
  C(t)=z_2(t)-z_1(t)  
\end{equation}
is called \emph{simple}.
We will study only cycles, as they are more natural than chains.

Consider a perturbation 
\begin{equation}\label{eq:var}
f(z) + \epsilon g(z)=t, 
\end{equation}
where $f, g \in \C[z]$. Let $z_i(t,\epsilon)$, for $\epsilon$ small, be solutions of \eqref{eq:var}, such that $z_i(t,0)=z_i(t)$, for $t$ a regular value of $f$.

A zero-cycle $C$ can be deformed by \eqref{eq:var} to a {\it one-parameter
family of zero-cycles}, that we will denote by 
\begin{equation}\label{eq:defC}
C_\epsilon(t) = \sum_{j=1}^m n_j z_j(t,\epsilon),
\end{equation}
and which we call the \emph{deformed cycle} $C_\epsilon(t)$. 
{ Note that if $\deg g>\deg f$, then $f+\epsilon g=t$ has more roots than $f$. These extra roots tend to infinity as $\epsilon\to0$. We define the deformed cycle by the same formula \eqref{eq:defC}, considering $n_{m+1}=\cdots=n_n=0.$}

In analogy with the case of 1-cycles on two-dimensional systems  (see \eqref{eq:D}) we define:

\begin{defi}\label{defint}
\begin{enumerate}[label=(\roman*)]\hfill
\item
The \emph{abelian integral of a polynomial function} $g\in\mathbb{C}[z]$ \emph{along a zero-cycle} $C$ (respectively along $C_\epsilon$) is the multivalued function 
which  associates
\[
\int_{C(t)}g(z) := \sum_{j=1}^m n_j g(z_j(t)),
\]
to $t$  belonging to
$\C\setminus\Sigma$, and
\[
\int_{C_\epsilon(t)}g(z) := \sum_{j=1}^m n_j g(z_j(t,\epsilon)),
\]
for $t\in\C\setminus\Sigma$, $|\epsilon|$ small. 
\item
The\emph{ displacement function}  of \eqref{eq:var} {\it along the perturbed family of zero-cycles} $C_\epsilon(t)$ is defined by
\begin{equation}\label{eq:Delta}
\Delta(t,\epsilon) := \int_{C_\epsilon(t)} f(z).
\end{equation}
\end{enumerate}
\end{defi}
Note that
\begin{equation}\label{eq:con_g}
\begin{aligned}
\Delta(t,\epsilon)&=\int_{C_\epsilon(t)} f(z)=\sum_{j=1}^m n_j f(z_j(t,\epsilon))=\sum_{j=1}^m n_j (t-\epsilon g(z_j(t,\epsilon))=\\
&=-\epsilon \int_{C_\epsilon(t)} g(z)=-\epsilon\int_{C(t)}g+o(\epsilon)=\epsilon M_1(t)+o(\epsilon),
\end{aligned}
\end{equation}
where we put 
\begin{equation}\label{eq:M1}
M_1(t)=-\int_{C(t)}g
\end{equation} and call it the \emph{first Melnikov function}. It is an \emph{abelian integral} on the zero-cycle $C$ and it gives the first order approximation of the displacement function.

Note that the function $M_1$ is analytic on regular values $\mathbb{C}\setminus\Sigma$ of $f$. 
Similarly, let $\Sigma_\epsilon\subset \C$ be the set of critical values of the polynomial $f+\epsilon g.$ The mapping $\Delta$ is analytic in 
$\C^2\setminus \left(\cup_{\epsilon\in\C}(\Sigma_{\epsilon}\times\{\epsilon\})\right).$

\bigskip

The two equations 
\begin{equation}\label{Delta=0}\Delta(t,\epsilon)=0,\quad \epsilon \text{ small}
\end{equation}and 
\begin{equation}\label{M1=0}
M_1(t)=0,
\end{equation}
for $\Delta$ given by \eqref{eq:Delta} and $M_1(t)$ given by \eqref{eq:M1}, lead to two problems: the \emph{infinitesimal} and  the \emph{tangential (16-th Hilbert problem).}

We say that $(t,\epsilon)\in \C^2\setminus \left(\cup_{\epsilon\in\C}(\Sigma_{\epsilon}\times\{\epsilon\})\right)$, 
is a solution of the infinitesimal problem if it verifies \eqref{Delta=0} and $t\in\C\setminus\Sigma$ is a solution of the tangential problem if it verifies \eqref{M1=0}. By abuse, we also say that the corresponding cycles $C_\epsilon(t)$ and $C(t)$ are solutions of the infinitesimal and tangential problem respectively. 
We then say that the $0$-dimensional \emph{limit cycle} $C_\epsilon(t)$ is born from the cycle $C(t)$ in the deformation \eqref{eq:var}.

We denote by $Z_\Delta$ the number of isolated solutions of the infinitesimal problem \eqref{Delta=0} and by $Z_1$ the number of isolated solutions $t\in\C\setminus\Sigma$ of the tangential problem \eqref{M1=0}. More precisely, we put 
$$
\label{Z_Delta}
Z_\Delta(f,g,C)=\inf_{e > 0}\left(\sup_{|\epsilon|<e}\# \left\{\text{ $t\in\C\setminus \Sigma_\epsilon$ isolated} \,|\,\Delta(t,\epsilon)=0 \right\}\right).
$$
This is inspired by the notion of \emph{cyclicity} given by Roussarie in \cite{Rou}.
More simply, 
\[
Z_1(f,g,C)=\# \{\text{ $t\in\C\setminus \Sigma$ isolated} \,|\,M_1(t)=0 \}.
\]
In both cases the vanishing means vanishing of at least one of the branches of the multivalued function $\Delta(t,\epsilon)$ or $M_1(t)$. These sets are finite,  since these two functions are algebraic and therefore have finitely many branches, unlike in 
 the $1$-dimensional case.


\medskip

 We count the lowest upper bound for the number of regular solutions in  $t$ of the infinitesimal or tangential problem respectively, for  \emph{any} cycle $C$.
    We denote it $Z_{\Delta}(f,g)$ and $Z_1(f,g)$ respectively. 

Finally, varying $f$ and $g$ of degree bounded by $m$ and $n$, respectively and the cycles $C$ of $f$, we define the numbers

\begin{equation*}
Z_\Delta(m,n)=\max_{\deg(f)\leq m, \deg g\leq n, C }\{Z_\Delta(f,g,C)\}.
\end{equation*}

\begin{equation}
    \label{eq:Z1}
    Z_1(m,n)=\max_{\deg(f)\leq m, \deg g\leq n, C }\{Z_1(f,g,C)\}.
\end{equation}

Here, we determine the two numbers  $Z_\Delta(m,n)$ and $Z_1(m,n)$, for any $m$ and $n$. 
Note that this solves the tangential and infinitesimal problem for zero-cycles. The corresponding problems in the context of one-cycles are far from being solved. Only a high (unrealistic) bound for the number of solutions of the tangential problem is given in \cite{BNY}. No general result for the infinitesimal problem on $1$-cycles is known. 


\begin{rema} \label{remGM}  
In \cite{GM} Gavrilov and Movasati studied (in our terminology) the \emph{tangential}  problem on \emph{simple} zero cycles. 
That is, they studied zeros of abelian integrals $M_1(t)=\int_{C(t)}g$,  of $g(z)\in\C[z]$, 
along \emph{simple cycles}  $C(t)=z_2(t)-z_1(t)$ of $f$ (with $z_1(t)\ne z_2(t)$). 
For $\deg(f)=m$, $\deg(g)=n$ they show that
\begin{equation}\label{Z_1}
 Z_1(f,g,C)\leq \frac{(m-1)(n-1)}{2}. 
 \end{equation}

They call it the infinitesimal problem but we prefer to reserve the term infinitesimal for the full problem of counting zeros of the displacement function $\Delta$ (see \eqref{Delta=0}).
We call tangential the  problem of counting zeros \eqref{M1=0} of the  first order term $M_1$ of the displacement function $\Delta$.

\medskip
Here we generalize their results in two directions: 
 
\begin{itemize}
\item[(i)] First, 
 we generalize their bound for the number of zeros in the \emph{tangential} 16-th Hilbert problem (i.e. zeros of $M_1$) along \emph{simple} cycles to the number of zeros on \emph{any} cycle (not necessarily simple).

\item[(ii)] 
Next, we generalize the results to the \emph{infinitesimal} (i.e. zeros of the displacement function $\Delta$) problem, for \emph{any} cycle $C_\epsilon(t)$. 
\end{itemize}

Moreover, we give optimal bounds for the two problems.

\end{rema}

Assume $M_1$ is not identically equal to zero. Then one can distinguish two types of cycles $C_\epsilon(t)$ solutions of the infinitesimal 16-th Hilbert problem: \emph{regular cycles} and \emph{alien limit cycles} (see \cite{CDR}). 

A cycle $C_\epsilon(t)$ is \emph{regular} if there exists a family of cycles $C_\epsilon(t(\epsilon))$ solutions of \eqref{Delta=0}, such that $C_0(t(0))$ is a solution of \eqref{M1=0}, with $t(0)\in\C\setminus \Sigma$. If not, it is called an \emph{alien} limit cycle. 

\begin{rema} 
From \eqref{eq:con_g}, 
it follows, by Rouch\'e's theorem,  that each regular solution of
\eqref{M1=0}
gives rise to a regular solution of
\eqref{Delta=0}, 
for $\epsilon$  small.  Hence, 

\begin{equation*}
 Z_1\leq Z_\Delta.
\end{equation*}
Given any regular cycle $C_0(t)$, solution of the tangential problem,
then by definition it corresponds to a cycle $C_\epsilon(t(\epsilon))$, solution of the infinitesimal problem. However, in general, not all solutions of the infinitesimal problem correspond to solutions of the tangential problem. 

In fact, 
$$
Z_\Delta=Z_1^m+Z_A,
$$
where $Z_1^m$ is the number of solutions of the tangential problem \emph{counted with multiplicity} and $Z_A$ is the number of \emph{alien} limit cycles in the deformation \eqref{eq:var}.
\end{rema}

\subsection{Motivation: classical infinitesimal and tangential 16-th Hilbert problems (on $1$-cycles)}

The classical infinitesimal 16-th Hilbert problem 
is the following problem: 

Given an integrable polynomial vector field $X$ in the plane with at least an annular region filled by its orbits, consider a small polynomial deformation. The infinitesimal problem studies the creation of limit cycles (i.e. isolated periodic solutions) in this deformation. 
The most important and most studied case is the case, when the deformation is of the form
\begin{equation}
\label{deffol}
dF+\epsilon\eta=0,
\end{equation}
where $F$ is a polynomial in two variables having a family of closed $1$-cycles $C(t)\subset F^{-1}(t)$ and $\eta$ is a polynomial 1-form. 
In order to study it, one studies the displacement function $D(t,\epsilon)$ on a transversal $T$.
The displacement function $D$ is defined by:
\begin{equation}\label{eq:D}
D(t,\epsilon)=\int_{C_\epsilon(t)}dF=-\epsilon\int_{C_\epsilon(t)}\eta=-\epsilon\int_{C_0(t)}\eta+o(\epsilon),  
\end{equation}
where $C_\epsilon(t)$ is a \emph{non-closed} cycle obtained by following the deformed foliation \eqref{deffol} starting from the point of $T$ belonging to $F^{-1}(t)$. 
 Isolated zeros of the displacement function $D$, are solutions of the classical infinitesimal problem for $1$-cycles and correspond to limit cycles  appearing in the deformation.

The first term $M_1(t)=-\int_{C_0(t)}\eta$ of $D$ is the \emph{first Melnikov function}. It is  an abelian integral along a $1$-cycle. Its zeros correspond to solutions of the classical tangential 16-th Hilbert problem on $1$-cycles. Note that these definitions motivate our Definition \ref{defint} for the displacement function $\Delta$ on zero cycles and the notation for its first order term $M_1$.

The problem can be studied in the real or complex plane. Note however, that, if one considers the problem for complex values of $t$, one has 
to restrict the domain of study to a simply connected domain. If not, the number can be infinite due to the possible presence of a logarithmic term. 
This is not the case in our $0$-dimensional case.

The infinitesimal and tangential 16-th Hilbert problems (for $1$-cycles) appear repeatedly in the list of Arnold's problem \cite{A}. It is one of the most recurrent problems in his list. See the extensive comment by S. Yakovenko to the problem 1978-6 (page 353) in \cite{A}.
Related problems are 1979-16,
1980-1, 1983-11, 1989-17, 1990-24, 1990-25, 1994-51 and 1994-52) in \cite{A}.

In \cite{A} problem 1994-51, Arnold poses the problem of polynomial deformations of integrable  vector fields in the plane having an annulus filled by closed orbits. He says: \emph{The location of the limit cycles
appearing in this perturbation is given in the first approximation by zeros of a certain integral (found by Poincar\'e) along non-perturbed closed curves (which are the
level curves of the first integral).
Is the number of zeros of the Poincar\'e integral bounded (by a constant
depending only on the degree of the perturbation)?}

In his next problem 1994-52, he says: \emph{A partial case of the previous problem: consider the full Abelian
integral
$I(h) = \int(Pdx+Qdy)$
along an oval of an algebraic curve $H(x,y) = h$. The polynomials $P(x,y)$ and
$Q(x,y)$ represent an infinitesimal variation of the Hamiltonian vector field, and
$I(h)$ is the Poincar\'e integral. Find an upper bound for the number of real zeros of
the function $I$ for all polynomials $(P, Q)$ of a fixed degree.}

In the book of Arnold's problems \cite{A}, problem 
 1983-11 reads:  \emph{Is it true that the integrals $I(h) = \int_{H=h}(Pdx+Qdy)$ with varying polynomials $P$, $Q$ form a Chebyshev system (or, at worst, the number of zeros is not
too much greater)?}
The question was answered negatively by examples by Rousseau and Zoladek \cite{RZ}. 
In a private communication Arnold conjectured then that \emph{ abelian integrals corresponding to  natural problems
form a Chebyshev system} i.e. the number of their zeros is strictly less than the dimension of the space. We call it \emph{Chebyshev property}. The precise mathematical notion of \emph{natural} was not given. Here we show that for $0$-cycles the Chebyshev property is far from being true.

\vskip0.5cm

Note that in the case of $1$-cycles by Caubergh, Dumortier, Roussarie \cite{CDR} showed that the infinitesimal and tangential problems are not equivalent. More precisely, they gave examples where there exist solutions of the infinitesimal problem not corresponding to solutions of the tangential problem. They call these extra solutions \emph{alien cycles}. We use the same terminology here for $0$-cycles. 

\vskip0.5cm
\subsection{Reducing the tangential problem on $1$-cycles to a problem on $0$ cycles} \label{reduction}
In  \cite{ABM}, the authors related the tangential 16-th Hilbert problem on $1$-cycles of deformations of hyperelliptic integrable systems with first integral $F(z,w)=w^2+P(z)$, with $P$ polynomial, to a generalization of the problem on zero dimensional cycles (see also \cite{CM}). 

More precisely, 
given a cycle $C(t)$ of $F(w,z)=t$ and a one-form $\omega=G(z,w)dz$, one studies the abelian integral $\int_{\gamma(t)}\omega$.
Solving $F(z,w)=t$, one gets $w=\sqrt{t-P(z)}$. Let $z_i(t)$ be the roots of $P(z)=t$ and let $g$ be the antiderivative of $G(z,w)dz$. One can assume that $\gamma(t)$ goes from $z_i$ to $z_j$ in one leaf of the Riemann surface and returns on the other leaf in the opposite direction, but also the opposite determination of $w$. The cycle might also be a sum of cycles of this type. Put $C(t)$ the simple zero cycle $z_j(t)-z_i(t)$
Then, taking the correct orientation, we have $\int_{\gamma(t)}\omega=2\int_{C(t)}g(z,w(z,t))$. So zeros of an abelian integral on $1$ cycles is reduced to a \emph{kind of} abelian integral on zero-cycles $C(t)$. It is not a true abelian integral, because the function $g$ is not polynomial, but a multivalued function. Its complexity is nevertheless bounded by the degree of the polynomial $G$.

In the same spirit, given a general polynomial first integral $F(z,w)$ in two-dimensional space, a polynomial form $\omega=G(z,w)dz$ and a cycle $\gamma(t)$, the cycle $\gamma(t)$ is given as a curve $w=w(z,t)$ on the Riemann surface $F(z,w)=t$. The Riemann surface has ramification points given as roots of the polynomial 
$\D(z,t)$, where $\D$ is the discriminant of $F(z,w)=t$ with respect to $w$. The curve $\gamma(t)$ can be described as a curve connecting certain roots $z_i(t)$ of $\D=0$. Denoting $g(t,z)$ the antiderivative (with respect to $z$) of $G(z,w(z,t))dz$ i.e $dg(t,z)=G(z,w(z,t))dz$, we get that 
\begin{equation}\label{gammaC}
\int_{\gamma}\omega=\int_{C(t)}g(z,t).  
\end{equation}
Here, $C(t)=\sum_{j=1}^m n_j z_j(t)$ is the zero-cycle of $\D=0$, obtained by taking all the ramification points $z_j(t)$ i.e. roots of the polynomial $\D(z,t)=0$, through which $\gamma(t)$ passes, with convenient signs. This is a kind of a generalization of abelian integrals on zero-cycles. As in the hyperelliptic case, of course, the function $g$ we integrate is not a polynomial. It is a multivalued function. Formula \eqref{gammaC} shows how the general tangential problem for $1$-cycles leads to a \emph{generalized} tangential problem on $0$-cycles. The same approach can be adapted to the infinitesimal problem, as well.

\section{Main Results}

We give the optimal bounds for the number of solutions of the tangential and infinitesimal problems on $0$-cycles in function of the degrees $m$ and $n$ of $f$ and $g$. {
Moreover, we describe the degeneracies where alien limit cycles can appear.}


The following theorem gives the optimal bound for the number of solutions  $Z_1(m,n)$ (defined by \eqref{eq:Z1})  of the tangential problem for any cycle $C$ of $f$. 

\begin{bigtheo} \label{theo:A}
Let $f,g$ be polynomials of degree $m$, $n$, respectively. 
\begin{itemize}
\item[(i)]
 If $m>2$ and $m$ does not divide $n$, then 
    $$
    Z_1(m,n)=n(m-1)!. 
    $$
\item[(ii)]
    If $m>2$ and $m$ divides $n$, then 
    $$
    Z_1(m,n)=(n-1)(m-1)!. 
    $$
   \item[(iii)] 
    If $m=2$, then 
    $$
    Z_1(m,n)=\left[\frac{n-1}{2}\right].
    $$
\end{itemize}
\end{bigtheo}

The following theorem gives the optimal bound for the number of solutions  $Z_{\Delta}(m,n)$ (defined by \eqref{Z_Delta})  of the infinitesimal problem for any cycle $C$ of $f$.

\begin{bigtheo}\label{theo:B}
Let $f,g$ be polynomials of degree $m$, $n$, respectively. 

     If $m>2$ and
        \begin{itemize}
            \item [(i)] if $n< m$, then $Z_\Delta(m,n)=n(m-1)!$.
            \item[(ii)] if $n\geq m$ then 
            \begin{itemize}
                \item[(a)] if $m$ does not divide $n$, then 
                $Z_\Delta(m,n)=\frac{m(n-1)!}{(n-m)!}.$
                \item[(b)] If $m$ divides $n$, then
                $Z_\Delta(m,n)=\frac{m(n-1)!}{(n-m)!}-(m-1)!.$
            \end{itemize}
        \end{itemize}
        
         If $m=2$,
         then $Z_\Delta(m,n)=\left[\frac{n-1}{2}\right]$.
\end{bigtheo}

Theorem~\ref{theo:B} follows directly from Propositions \ref{n<m} and \ref{n>m}.
The following theorem gives  conditions on the degrees of $f$ and $g$ under which the tangential and the infinitesimal problem are not  equivalent i.e. under which
\emph{alien limit cycles} can exist. Recall that alien limit cycles are solutions of the infinitesimal problem, not corresponding to solutions of the tangential problem.  

\begin{bigtheo}\label{theo:C}
Let $f,g$ be polynomials of degree $m$, $n$, respectively. 
\begin{itemize}
    \item [(i)] If $m>2$, $n<m$, then for generic cycles $C$,  and $f$ and $g$ generic, there can exist no \emph{alien} cycles.
    \item[(ii)] If $m>2$, $n> m$, and $m$ does not divide $n$, then for generic cycles $C$, there exist \emph{alien} limit cycles for generic  $f$, $g$.
    \item[(iii)] If $m=2$, for regular at infinity cycle $C$,  and $g$ generic, there are no \emph{alien} limit cycles.
\end{itemize}
\end{bigtheo}
Theorem~\ref{theo:C} follows directly by comparison from Theorems~\ref{theo:A} and~\ref{theo:B} and the Remark \ref{gen}.

\begin{rema}
    In the above theorems the optimal bounds are obtained for a \emph{generic} choice of $f$ and $g$
 and for \emph{generic} cycles $C$ in a precise sense introduced below: 
 
 For $m>2$, we ask the cycle $C$ to be regular  at infinity (Definition \ref{mgen}) and asymmetric (Definition \ref{asym}). For $m=2$, the cycle is necessarily simple \eqref{simple} (hence symmetric), but we ask it to be regular at infinity.

 \end{rema}


{ As recalled in the Introduction, it was first conjectured by Arnold that abelian integrals corresponding to natural deformations form Chebyshev systems. Later some counter examples were given \cite{RZ}. 
However, in these examples the discrepancy was small.

\begin{rema}\label{remCheb}
In \cite{GM} Gavrilov and Movasati studied abelian integrals on \emph{simple zero cycles} (recall Remark \ref{remGM}). Denote  $\mathcal{P}^S(m,n)$ the dimension of the space of abelian integrals
along simple zero cycles, for $\deg f=m$, $\deg g=n$, {$C$} 
a simple cycle of $f$ and $Z_1^S(m,n,C)$ the maximal number of zeros of these abelian integrals. In \cite{GM}, they estimated the above numbers and conjectured that 
\begin{equation}\label{limit}
\lim_{m\to\infty}\frac{Z_1^S(m,m-1)}{\dim \mathcal{P}^S(m,m-1)}=1. 
\end{equation}
This weaker version of Arnold's conjecture could be called \emph{asymptotic Chebyshev conjecture for abelian integrals on simple zero cycles.}
    \end{rema}

Denoting $\mathcal{P}(m,n)$ the dimension of the space of abelian integrals for the same deformation if any cycle $C$ is considered, then it follows directly from Theorem~\ref{theo:A}  that 

\begin{bigtheo}\label{theo:D} \hfill
\begin{itemize}
    \item [(i)] Considering any cycle, we have
    $$
    \lim_{m\to\infty}\frac{Z_1(m,m-1)}{\dim\mathcal{P}(m,m-1)}=\infty.
    $$
   \item[(ii)]  Considering only simple cycles, we have
    $$
    \lim_{m\to\infty}\frac{Z_1^S(m,m-1)}{\dim\mathcal{P}^S(m,m-1)}=\infty.
    $$
    \end{itemize}
\end{bigtheo}
Claim (i) follows from Theorem~\ref{theo:A} and the calculation of the dimension of the dimension of the  Brieskorn modulus in Proposition \ref{prop:brieskorn}. 
Similarly, (ii) follows from Theorem \ref{th:simpletang} and the same calculations of the Brieskorn modulus.

The same holds if $\deg g=m$ instead of $m-1$.
This shows that the discrepancy from the Chebyshev property for abelian integrals on zero cycles is very big. 
That is, the \emph{asymptotic Chebyshev conjecture for abelian integrals on zero cycles} \eqref{limit} does not hold.

Let us note also that in Theorem \ref{th:simpletang}, we determine the optimal bound for the number of zeros of abelian integrals along simple cycles $Z_1^S(m,n)$, for $\deg(f)=m$, $\deg(g)=n$. This number coincides with the bound obtained by Gavrilov and Movasati if $m$ and $n$ are coprime. However, our bound is slightly better in the exceptional cases when $m$ and $n$ are not coprime. }

One can also consider the analogous asymptotic question in the infinitesimal problem instead of the tangential.
Note however, that the space of deformations in the infinitesimal problem is not a vector space, but a variety. 
The dimension of this variety is equal to its tangent vector space, which is precisely given by the dimension $\mathcal{P}(m,m-1)$.
On the other hand, considering the infinitesimal problem would only increase the numerator $Z_\Delta(m,m-1)$ in Theorem~\ref{theo:D}, so the same limit holds for the infinitesimal problem, as well.

\section{Connection curve $\Gamma_f$ and the zero hypersurface $S_g$}
In this section, we introduce the machinery that we will use in the study of zeros of abelian integrals $M_1(t)=\int_{C(t)}g$ on zero cycles $C(t)$ of $f$.

First we show that, by a simple reduction, we can always reduce the tangential problem to the problem where $m$ does not divide  $n$.

\begin{lema}\label{lem:reduction}
Given two polynomials in one variable $f$ and $g$ of degrees $m$ and $n$ and a $0$-cycle $C$ of $f$.
Let $M_1=\int_{C}g$ be the abelian integral of $g$ on the cycle $C$. 

Then, there exists a polynomial $\tilde g$ such that $\tilde n=\deg\tilde g$ is not a multiple of $m=\deg f$ and such that 
\begin{equation*}
\int_{C}g=\int_{C}\tilde g.
\end{equation*}   
\end{lema}

\begin{proof}
    Note first that, from the assumption that $C$ is a cycle of $f$, it follows that $\int_{C}af^k=0$, for any $k\in\N$ and $a\in\C$. Indeed, $\int_{C(t)}af^k=a\sum_{i=1}^m n_if^k(z_i(t))=a\sum_{i=1}^m n_it^k=at^k\sum_{i=1}^m n_i=0.$ This is not true if $a$ is a non-constant polynomial. 

    If $n=k_1m$, then there exists $a_1\in\C$,  $\tilde n_1<n$ in $\N$, such that $g=a_1f^{k_1}+\tilde g_1$,  and hence $\int_{C}g=\int_{C}\tilde g_1$. If $\tilde n_1$ is not a multiple of $m$, we are done. If not, we repeat the procedure  with $\tilde g_1$, until obtaining
    $g=\sum a_i f^{k_i}+\tilde g$, with $\tilde n=\deg\tilde g$ not a multiple of $\deg f$ and $\int_{C}g=\int_{C}\tilde g$.  The procedure stops, as at each step, we reduce the degree of $\tilde g_i$.  
\end{proof}

Thanks to this Lemma, in the sequel of this section, without loss of generality, we can assume that $\deg g=n$ is not a multiple of $m=\deg f$. 

Note also, that for generic $g$ and $f$, for $m$ dividing $n$ after one reduction, we will end up with $\tilde g$ of degree 
$\tilde n=n-1$, such that $m$ does not divide $\tilde n$.

In order to prove Theorem~\ref{theo:A}, we define two algebraic sets in $\CP^m$: a curve $\Gamma_f$, which we call \emph{connection curve} associated to $f$ and a hypersurface $S_g$, which we call \emph{zero hypersurface} associated to $g$ and the weights $n_i$ given by the cycle $C$. The zeros of the abelian integral $M_1=\int_{C}g$ correspond to intersections of the two sets. We calculate the degree of the curve $\Gamma_f$ in Lemma \ref{degGamma}. The degree of $S_g$ is trivially equal to $n=\deg g$. 

By Bezout's theorem, the number of intersection points counted with multiplicity (including points at infinity) is given by the product of the two degrees. This product gives the bound $Z_1(f,g,C)\leq n(m-1)!$


In order to prove that the bound is attained for $m>2$, for convenient $f$, $g$ recall that we count only regular values of $t$ (where the multiplicity is just one). Moreover, we don't count the zeros at infinity, which correspond to intersections of $\Gamma_f$ and $S_g$ at infinity.

We prove that there exists a generic choice of $f$ and $g$ 
{ with \emph{minimal number} of intersection points at infinity. In fact, there are no such points if $n<m$, or $n\geq m$, with $m$ not dividing $n$. Moreover, all intersection points correspond to regular values of $t$. }
More precisely, we count values of $t$ corresponding to intersection points $(z_1:\ldots:z_{m+1})\in\CP^m$. 
We show that generically, each intersection point $(z_1:\ldots:z_{m+1})$ of $\Gamma_f$ and $S_g$ corresponds to only one value of $t$. 
Hence, for such $f$ and $g$, the number given by Bezout's theorem gives the exact bound for the number of zeros of the first order Melnikov function $M_1$ and (i) of Theorem~\ref{theo:A} will follow. 

The case $m=2$ is exceptional. In that case, we show (Lemma \ref{lem:irreducibility}) that the zero hypersurface $S_g$ is reducible. Taking away the trivial hyperplane component reduces the degree of the \emph{reduced zero hypersurface} $\tilde S_g$.

\subsection{The \emph{connection curve} $\Gamma_f$}
We start with the set \[\tilde{\Gamma}_f=\{(z_1,...,z_m) \,\vert\,f(z_i)=f(z_j)\}\subset \C^m.\]
We define the curve $\Gamma_f\subset \CP^m$  as the closure of $\tilde{\Gamma}_f\setminus \cup_{i\not= j}\Delta_{ij}$ in $\CP^m$, where $\Delta_{ij}=\{z_i=z_j\}$ are hyperplane diagonals (giving trivial solutions of $f(z_i)=f(z_j)$). Note that the complex curve $\Gamma_f$ is parametrized by the values $t$ of $f(z_1)=\cdots=f(z_m)=t\in\CP^1.$ We call $\Gamma_f$ the \emph{connection curve}, as together with its parametrization, it contains the same information as the \emph{Gauss-Manin connection on zero cycles}.
\begin{lema}\label{degGamma}
The degree of the connection curve $\Gamma_f$ is equal to $(m-1)!$
\end{lema}

\begin{proof} 
Let $Sym_m$ be the group of permutations of $(1,\ldots,m)$ and $Stab_m\subset Sym_m$ its subgroup preserving $m$ i.e. 
the group of permutations $\alpha=(\alpha_1,...,\alpha_{m-1},\alpha_m)$, with $\alpha_m=m$.
Let $\xi=e^{\frac{2\pi i}{m}}$, be an $m$-th root of unity.

The degree of a curve in $\CP^m$ is given as the number of intersection points (counted with multiplicity) with any hyperplane. Here, we consider intersections of $\Gamma_f$ with the hyperplane at infinity $L_\infty$. 
It consists of points 
$$p_\alpha=\left(\xi^{\alpha_1}:...:\xi^{\alpha_{m-1}}:\xi^{\alpha_m}:0\right)\in\CP^m,$$ 
where $\alpha=(\alpha_1,...,\alpha_{m-1},\alpha_m=m)$
belongs to the stabilizer 
$Stab_m\subset Sym_m$ of $m$.

Indeed, the set of solutions of the equation $f(z)=t$ tends to a rescaled set of roots of unity of degree $m$ as ${t}\to\infty$.  The last homogeneous coordinate of $p_\alpha$ are $0$, as  points belong to the hyperplane at infinity and by scaling (working in $\CP^m$), we can assume that the $m$-th coordinate equals $1$. Thus, 
the intersection of $\Gamma_f$ with hyperplane  $L_\infty\subset \CP^m$ at infinity 
consists of points 
$\left(\xi^{\alpha_1}:...:\xi^{\alpha_{m-1}}:1:0\right)$, with $\alpha\in Stab_m$ as above. 

This intersection is transversal, and as $|Stab_m|=(m-1)!$ is the number of permutations of $m-1$ points, it follows that  $\deg\Gamma_f=(m-1)!$
\end{proof} 

\medskip

An alternative proof can be obtained by using the Vieta mapping $V:\C^m\to\C^m$, given by
\begin{equation}
    \label{Vieta}
    V(z_1,\cdots,z_m)=(\sigma_1,\cdots,\sigma_m), \quad \sigma_k=(-1)^k\sum_{1\leq i_1<i_2<\cdots<i_k\leq m}z_{i_1}\cdots z_{i_k}.
\end{equation}

Recall that $(z_1,\ldots,z_m)$ are roots of the polynomial 
\begin{equation*}
f_\sigma(z)=z^m+\sum_{i=1}^m\sigma_i z^{m-i}.
\end{equation*}
Denoting 
$\sigma^{m-1}=(\sigma_1,\ldots,\sigma_{m-1})$, 
$\sigma=(\sigma^{m-1},-t)\in\C^{m-1}\times\C$ and $f_{\sigma^{m-1}}=f_{(\sigma^{m-1},0)}$, then we have that 
$z_1(t),\ldots,z_m(t)$ are solutions of the equation $f_{\sigma^{m-1}}=t$.
In other words, 
$\Gamma_f=\left(\pi_{m-1}\circ V\right)^{-1}\left(\sigma^{m-1}\right)$, where $\pi_{m-1}:\C^m\to\C^{m-1}$ { is the canonical} projection to the first $(m-1)$ variables.
That is, $\Gamma_f$ is given as the solution of the equations \eqref{Vieta}, for $i=1,\ldots,m-1$.
Hence,  $\deg\Gamma_f=1\cdot 2\cdots (m-1)$ by Bezout's theorem.


\begin{rema}
Let $t$ be a non-critical value of $f$. The monodromy of $f$ permutes the points  of $\{f=t\}$. It thus can be identified with a subgroup $Mon_f$ of $Sym_m$. 
The connected components of $\Gamma_f$ are in one-to-one correspondence with the orbits of the action of $Mon_f$ on $Sym_m$ by left multiplication. Thus $\Gamma_f$ has $\frac{m!}{|Mon_f|}$ irreducible components.
For example, if $f$ is a Morse polynomial, then $\Gamma_f$ is an irreducible curve.\end{rema}

\subsection{The \emph{zero hypersurface} $S_g$}


\medskip

We define  the \emph{zero hypersurface} $S_g$ as the closure of the algebraic hypersurface $\{G=0\}$ in $\CP^m$, where 
\begin{equation*}
G(z_1,\ldots,z_m)=\sum_{j=1}^m n_j g(z_j).
\end{equation*}
Clearly, $\deg S_g=n=\deg g$.

\begin{lema}\label{lem:irreducibility}
\hfill
\begin{itemize}
    \item[(i)] If the cycle $C$ is not simple (i.e. at least three $n_j$ are non-zero),  then $ S_g$ is irreducible. In particular, $\Delta_{ij}\not\subset S_g$.
    \item[(ii)] If $\Delta_{ij}\subset S_g$, then $n_i+n_j=0$ and $n_k=0$, for $k\not=i,j$. This means that the cycle $C=\sum n_i z_i$ is a simple cycle.
    \item[(iii)] For a simple cycle $C=z_i-z_j$, the hypersurface $S_g$ is a union of  the diagonal $\Delta_{ij}$ and a hypersurface $\tilde{S}_g=\{\tilde{g}_{ij}=0\}$ of degree $n-1$, where $\tilde{g}_{ij}=\frac{g(z_i)-g(z_j)}{z_i-z_j}$. 
\end{itemize}
\end{lema}

\begin{proof}
(i) Assume that $n_i\not=0$, for $i=1,2,3$, and consider the intersection $ S_g\cap L_\infty\cap\{z_4=\dots=z_m=0\}$. This is a projective curve 
$$\sigma=\{n_1z_1^n+n_2z_2^n+n_3z_3^n=0\}\subset\CP^2=L_\infty\cap\{z_4=\dots=z_m=0\}.$$ 
In a suitable affine chart  $\sigma=\{x^n+y^n=1\}$, i.e. it is a generic level curve of the polynomial $x^n+y^n$. Therefore, as any generic level curve of a polynomial, it is  irreducible.
Thus, any irreducible component of $ S_g$ must contain $\sigma$. But  $\deg\sigma=\deg S_g$,
which implies that $ S_g$ is irreducible.

(ii)  Let us differentiate $n$  times $G$ along the vector field $v_{ij}=\partial_i+\partial_j$ tangent to $\Delta_{ij}$. As $G$ vanishes identically on $\Delta_{ij}$, its $n$-th  derivative vanishes identically as well. However,  $\left(L_{v_{ij}}\right)^nG\equiv n!(n_i+n_j)$, so $n_i+n_j=0$, and therefore  $G\vert_{\Delta_{ij}}=\sum_{k\not=i,j}n_kg(z_k)\equiv0$. It follows $n_k=0$, for $k\ne i,j$, as the coordinates $z_k$, $k\ne i$, form  a system of coordinates on $\Delta_{ij}$.

(iii) In this case $S_g=\{g(z_i)-g(z_j)=0\}$. 
If $g(z)=\sum a_kz^k$, then 
$$
g(z_i)-g(z_j)=(z_i-z_j)\tilde{g}_{ij}, \quad\text{where }\tilde{g}_{ij}=\sum a_k\frac{z_i^k-z_j^k}{z_i-z_j},
$$
and the claim follows.

\end{proof}

\begin{rema}
 It follows from the above Lemma that, if a cycle $C$ is not simple, then the zero hypersurface $S_g$ is irreducible. Simple cycles are an exceptional case, in which the diagonal $\Delta_{ij}$, corresponding to the simple cycle, is an irreducible  component of $S_g$.    
\end{rema}

\section{Solution of the tangential problem}

The aim of this section is to prove Theorem~\ref{theo:A}. Recall that we are counting the number of isolated zeros of $M_1(t)$, which depends on the coefficients of the polynomials $f,g$. In particular, varying the coefficients, the number of isolated zeros can only increase, so, along the section, we will assume that $f,g$ are generic polynomials.

\medskip 

We need in addition some kind of genericity condition on the cycle. Note that we are assuming the coefficients to be integer numbers, so we need to precise this a little bit.  

\medskip

The first condition on the cycle $C$ we impose is not being simple. 

\begin{lema}\label{trans}
Let $C$ be a non-simple cycle of $f$.  Then the intersection of  $\Gamma_f$ and $S_g$ outside $L_\infty$ is transversal and lies outside $\cup_{ij}\Delta_{ij}$.
\end{lema}
\begin{proof}

The image $V(\Gamma_f)$ of $\Gamma_f$ under the Vieta mapping defined in \eqref{Vieta}   is the line $\pi^{-1}_{m-1}({\sigma^{m-1}})=\{{\sigma^{m-1}}\} \times \C$. 

By Lemma~\ref{lem:irreducibility},  $\Delta_{ij}\not\subset S_g$, for any $1\le i\not=j\le m$, for a {non simple} cycle $C$. Thus, the intersection  $ S_{g\Delta}= S_g\cap\bigcup\Delta_{ij}$ is an algebraic set of dimension at most $m-2$. Therefore the set $V( S_{g\Delta})$ also has dimension at most $m-2$, i.e. is of codimension at least $2$.

Thus, for a generic  value $\{\sigma^{m-1}\}$ and $f=f_{\sigma^{m-1}}$, the line $V(\Gamma_f)$ does not intersect $V( S_{g\Delta})$, which implies $\Gamma_f\cap S_g\cap\bigcup\Delta_{ij}=\emptyset$.
\end{proof}

\begin{defi}\label{mgen}
    Given polynomials $p$ and  $q$ of degrees  $k$ and $\ell$, and a cycle $K$ of $p$, consider the connection curve $\Gamma_{p}$ and the zero hypersurface $S_{q}$. We say that the cycle $K$ is \emph{regular at infinity}, if the number of points at infinity $\Gamma_{p}\cap S_{q}\cap L_\infty$ is minimal among all the cycles of $p$. 
\end{defi}
We show that the regularity of a cycle depends only on the degrees $k$ and $\ell$ of $p$ and $q$ and not on the polynomials themselves.


\begin{lema}\label{noinf}
Let $C$ be a cycle of $f$. If $m=\deg f$ does not divide $n=\deg g$, then the intersection $\Gamma_f\cap S_g\cap L_\infty$ is empty (i.e. the cycle $C$ is regular at infinity),  if and only if, the weights $n_j$ of the cycle $C=\sum n_j z_j$  verify a finite number of inequations \eqref{eq:generic}, compatible with the cycle condition $\sum n_j=0$:

\begin{equation}\label{eq:generic}
  \sum_{j=1}^m n_j\xi^{n\alpha_j}\ne0, \quad {\text{for }} \xi=e^{2\pi i/m}  \quad\text{and any } \quad\alpha\in Stab_m \subset Sym_m,
\end{equation} 
where $\alpha=(\alpha_1,...,\alpha_{m-1},\alpha_m=m)$. 
\end{lema}

\begin{proof}
Restriction of $G$ to $L_\infty$ equals $G_\infty=\sum n_j z_j^n$.
Recall that the points of $\Gamma_f$ at infinity are 
the points $p_\alpha=\left(\xi^{\alpha_1}:...:\xi^{\alpha_{m-1}}:1:0\right)\in L_\infty$. Hence, $\Gamma_f\cap S_g\cap L_\infty=\emptyset$, if and only if $G_\infty(p_\alpha)\ne0$, for all $\alpha=(\alpha_1,\dots,\alpha_{m-1},m)\in Stab_m \subset Sym_m$. This is exactly the condition \eqref{eq:generic}, and this condition is a generic condition on the space of cycles  as soon as $n$ is not divisible by $m$. 
\end{proof}
Note that, if $m$ divides $n$, then the set $\Gamma_f\cap S_g\cap L_\infty$ is non-empty. Indeed, in that case all $\xi^{n\alpha_j}=1$ and $\sum_{j=1}^m n_j\xi^{n\alpha_j}=\sum_{j=1}^m n_j=0$, by the cycle condition, so condition \eqref{eq:generic} cannot be verified.
Moreover, by Lemma \ref{lem:reduction}, we can reduce $g$ to $\tilde g$, of degree $\tilde n$, with $m$ not dividing $\tilde n$. That is, the cycle $C$ is regular at infinity with respect to $f$ and $\tilde g$. We will see, when dealing with the infinitesimal problem, that a cycle regular at infinity will not necessarily correspond to no intersection points of $\Gamma_p\cap S_q \cap L_\infty$.

{\color{black}
We need  to count the number of regular finite values of $t$ such that $f(z_1)=\cdots=f(z_m)=t$ and $\sum n_i g(z_i(t))=0$. However, we  rather count the corresponding points $(z_1:\ldots:z_m:1)\in\CP^m$. The problem is that, in general, various points $(z_1:\ldots:z_m:1)$ can correspond to the same value of $t$. This depends on the symmetries of the cycle $C$.

\begin{defi}\hfill
\begin{itemize}
    \item[(i)]
Let $Sym_m$ denote the group of permutations of $(1,\ldots,m)$. Let $C=\sum_{j=1}^m n_j z_j$ be a cycle. Let $H\subset Sym_m$ be a subgroup of $Sym_m$ preserving $(n_1,\ldots,n_m)$ up to sign i.e. 
\begin{equation}\label{sym}
h(n_1,\ldots,n_m)=(n_1,\ldots,n_m), \quad \text{or}\quad  h(n_1,\ldots,n_m)=-(n_1,\ldots,n_m). 
\end{equation}
We call it the \emph{symmetry group of the cycle} $C$.
Given any $h\in Sym_m$ it acts on a cycle $C$ by $h(C)=\sum_{j=1}^m n_{h(j)} z_j=\sum_{j=1}^m n_{j} z_{h^{-1}(j)}$. 

\item[(ii)]
We say that a cycle $C=\sum_{i=1}^m n_i z_i$ is \emph{symmetric} if its symmetry group $H$ is non-trivial and \emph{asymmetric} if it is trivial.
\end{itemize}
\end{defi}

Note that if $h$ belongs to the symmetry group $H$ of a cycle $C$, then  if $t$ verifies $\int_{C(t)}g=0$, then it also verifies $\int_{h(C(t))}g=\pm\int_{C(t)}g=0$. Hence, if   $(z_1:\ldots:z_m:1)$ is an intersection point of $\Gamma_f$ with $S_g$, then all permutations by $h^{-1}\in H$ of $(z_1:\ldots:z_m:1)$, correspond to the same solution $t$ of the tangential problem.

By genericity, we will assume that $f$ is indecomposable. Let us show that 
if $h\in Sym_m\setminus H$, then the two functions $\int_{h(C(t))}g$ and $\int_{C(t)}g$, do not coincide (up to sign), for $g$ generic. Taking the difference (or sum) of the two cycles, the problem amounts to showing that for $g$ generic, for any cycle $C$, $\int_C g\equiv0$ implies $C=0$. But this follows from the solution of the tangential center problem (see Theorem 2.2 of \cite{ABM}, or Proposition 3.1 of \cite{GP}). Indeed, these results show that it happens only if $g=P(f)$ for some polynomial $P$.  That is for a non-generic polynomial $g$. \color{black}

It can happen nevertheless, that for some particular value of $t$ both integrals vanish.  We will show that \color{black} this can be broken by a small deformation of $f$ and $g$. 

\begin{exam}\hfill
\begin{itemize}
\item[(i)] Any simple cycle $C=z_1-z_2$ is symmetric. 

\item[(ii)] The cycle $C=z_1-z_2+z_3-z_4$ is a symmetric cycle, but is not simple.

\item[(iii)]  Any cycle having the property that $|n_i|\ne|n_j|$, for $i\ne j$ is asymmetric, but 
a cycle can be asymmetric without verifying this property.

\item[(iv)]    An explicit example of an asymmetric cycle for $f$ of any degree $m>2$ is as follows:

If $m=2\ell$ is even,
put $n_{2i}=2i$, $i=1,\ldots,\ell$,
$n_{2i-1}=-(2i-1)$, $i=1,\ldots,\ell-1$ and $n_{2\ell-1}=-3\ell+1$. 

If $m=2\ell+1$ is odd, put
$n_{2i}=2i$, $i=1,\ldots,\ell-1$, $n_{2i+2}=3\ell+1$,
$n_{2i-1}=-(2i-1)$, $i=1,\ldots,\ell+1$. 

In either case, one verifies that $\sum_{i=1}^m n_i=0$ and that 
 $|n_i|\ne|n_j|,$ for $i\ne j$, as $m>2$, so $\ell>1$.
 We give here the proof of the first property for $m$ even:
 $$\sum_{i=1}^m n_i=\sum_{i=1}^\ell n_{2i}+\sum_{i=1}^\ell n_{2i-1}=\ell(\ell+1)-[(\ell-1)^2]+(-3\ell+1).$$

\end{itemize}
\end{exam}

\begin{lema}\label{asym}
Let $\deg f=m$. Then, for any $m>2$, there exist asymmetric cycles of $f$. 
More precisely, in the space of cycles, it is a complement of a finite number of hyperplanes. 
For $m=2$, any cycle of $f$ is simple and hence symmetric.   
\end{lema}

\begin{proof}

An asymmetric cycle is given as a simultaneous solution of the following conditions:
$\sum_{i=1}^m n_i=0$, and $n_i\ne n_j$, and $n_i\ne -n_j$, for $i\ne j$. That is, from the hyperplane ${\mathcal C}$ given by the cycle condition, we eliminate points belonging to a finite number of hyperplanes, as in the notion of genericity. There remains an open dense set for $n_i\in\R$, unless one of the hyperplanes we eliminate coincides with the hyperplane ${\mathcal C}$. This occurs precisely in the case $m=2$. 
\end{proof}

\begin{rema}\label{gen}   
In general, \emph{generic}  means \emph{belonging to an open dense set}. Here, we will speak about genericity in the space of cycles. Recall that a cycle is given by integer weights $n_j$. 

We impose two conditions on the cycles: \emph{regularity at infinity} and \emph{asymmetry}. They are given by conditions
\eqref{eq:generic} and \eqref{sym} corresponding to a finite number of linear inequation conditions, which have to be \emph{compatible with the cycle condition} \eqref{cycle}. 
If for some $\alpha$, the condition \eqref{eq:generic} is incompatible with the cycle condition, we don't apply it and this results in increasing the minimal bound in the definition of regularity at infinity. 
The Zarisky closure of the set of generic $n_j$ in $\R^m$ is an open dense set, namely a complement of a finite union of hyperplanes.
Therefore, the two conditions on the cycles: regularity at infinity and asymmetry are compatible generic conditions. 
\end{rema}

\begin{rema}
  \label{gen asym}
  It follows from Remark \ref{gen} and Lemma \ref{asym} that for $m>2$ and any $n<m$, there exist regular at infinity  asymmetric cycles of $f$ of degree $m$.      
  \end{rema}

\begin{lema}\label{differentt}
For an asymmetric cycle $C$ of $f$, each point $(z_1:\cdots:z_{m}:1)\in \Gamma_f\cap S_g$ lies on a different fiber $f^{-1}(t)$, $t\in\C$.
\end{lema}

\begin{proof}
\color{black}
Let us suppose that there are two such points $(z_1:\cdots : z_{m}:1)$ and its permutation $(z_{\sigma(1)}:\cdots: z_{\sigma(m)}:1)$, which lie on the same fiber $f^{-1}(t_0).$ 
As the cycle is asymmetric, the permutation $\sigma$ is not a symmetry. Hence, the functions $I_1(t)=\int_{C(t)}g$ and $I_2(t)=\int_{\sigma^{-1}(C(t))}g$ do not coincide identically, but $I_1(t_0)=I_2(t_0)=0$, for $t_0$ regular.
Making a small deformation of $f$ and $g$, we can assume that $t_0$ is a simple zero of $I_1$. Next, we make a second small deformation of $f$ and $g$, such that $I_1(t_0)=0$, but $I_1(t_0)-I_2(t_0)\ne0$.
This proves the claim.

\end{proof}

\begin{proof} [Proof of Theorem~\ref{theo:A}] The intersection points of the connection curve $\Gamma_f$ and the zero hypersurface $S_g$
count with multiplicity all the values $(z_1:\cdots:z_{m+1})\in\CP^m$ such that $f(z_1)=\cdots=f(z_m)=t$ and $\sum n_j g(z_j)=0$.
As $\deg\Gamma_f=(m-1)!$ (By Lemma \ref{degGamma}) and $\deg S_g=n$, it follows by Bezout's theorem, that the number of intersection points of $\Gamma_f$ and $S_g$ is equal to $n(m-1)!$ This gives an upper bound for the number of zeros $t$ of the first nonzero Melnikov function $M_1$ and proves the bound $Z_1(f,g,C)\leq n(m-1)!$. 

In order to prove (i), note first that Bezout's theorem gives the \emph{exact number} of intersection points of two algebraic varieties in the projective space counted with multiplicity. Note however, that we have to count only points at finite distance and belonging to regular fibers. Moreover, we do not count the intersection points $(z_1:\cdots:z_{m+1})\in\CP^m$ themselves, but rather the corresponding values $t=f(z_1)=\cdots=f(z_m)\in \C$. 
Now assume that $f,g$ are   generic, 
so that Lemma~\ref{trans} applies and moreover take a regular at infinity asymmetric cycle $C$, so that Lemma~\ref{noinf} applies, as well. 
Then claim (i) follows, as by  Lemma~\ref{noinf}, there are no intersection points at infinity and by Lemma~\ref{trans},  all intersection points have multiplicity $1$. 

Finally, 
by Lemma \ref{differentt},
making an additional small deformation of $f$, we obtain that all above intersection points of $\Gamma_f$ and $S_g$ lie on different fibers $f^{-1}(t)$, so that under genericity hypothesis, counting intersection points of $\Gamma_f$ and $S_g$ in $\CP^m$ is the same as counting their corresponding values $t$. This finishes the proof of (i) of Theorem~\ref{theo:A}.

In order to study the exceptional case $m=2$, note first that, if $m=2$, then 
by Lemma~\ref{lem:reduction}, we can assume that $n$ is odd. Next, note that, if $m=2$, then
necessarily the cycle $C$ is a simple cycle. 
In that case we can apply \eqref{Z_1} from \cite{GM}, giving the bound 
(ii), of the Theorem.
Note that the bound is an integer, as $n$ is odd. 
One proves the optimality of the bound as in the case $m>2$, but using $\tilde S_g$, as introduced in (iii) Lemma \ref{lem:irreducibility}.
\end{proof}

\section{Solution of the infinitesimal problem}

Recall that in the infinitesimal problem, we study zeros of the displacement function $\Delta(t,\epsilon)=\int_{C_\epsilon}f.$

The study of its zeros, follows the same general lines as the study of the tangential problem, recalling that 
one studies the cycles of $f+\epsilon g$ instead of cycles of $f$. Hence, we work with the connection curve $\Gamma_{f+\epsilon g}$. On the other hand, we study the integral of $f$ instead of the integral of $g$. Hence, we study finite regular intersection points of the curve $\Gamma_{f+\epsilon g}$ with the zero hypersurface $S_f$. However, 

\begin{rema}\label{rem:fg}
  The integral $\int_{C_\epsilon}f=0$, if and only if,  $\int_{C_\epsilon}g=0$, for $\epsilon\ne0$. 
\end{rema}
Indeed,  
\begin{equation*}
   \int_{C_\epsilon}f+\epsilon g\equiv0. 
\end{equation*}

Hence, when convenient, instead of considering intersections of $\Gamma_{f+\epsilon g}$ with the zero hypersurface $S_f$, we can consider intersections of $\Gamma_{f+\epsilon g}$ with the zero hypersurface $S_g$.



Note also that, contrary to the situation in the tangential problem, where the deformation enters 
linearly (in the integrand) in the problem, in the infinitesimal problem the deformation enters in a nonlinear way through the cycle 
 $C_\epsilon$. 
 
 One cannot perform the reduction of the degree of $g$, as in Lemma~\ref{lem:reduction}. 
 
 The behavior at infinity is given by the leading term of $f+\epsilon g$, for $\epsilon\ne0$. The degree of $f+\epsilon g$, for $\epsilon\ne0$ small, is always at least equal to the degree of $f.$ There are two cases: 

\begin{enumerate}
    \item 
if $n=\deg g\textcolor{black}{<}\deg f=m$, then $\deg (f+\epsilon g)=\deg f=m$.
\item
if $n=\deg g\textcolor{black}{\geq }\deg f=m$, then $\deg (f+\epsilon g)=\deg g=n$ and then arithmetic properties of $m$ and $n$ come into the play. 
\end{enumerate}
{\color{black} If $\deg g> \deg f$, then we deal with \emph{singular perturbations}. This is the reason why this case is more complicated.}
\begin{prop}\label{n<m}
    
Let $f,g$ have degree $m,n$, respectively.
If \textcolor{black}{$n < m$}, then
    \begin{itemize}
    \item[(i)] For  any cycle $C$, 
    \[Z_\Delta(f,g,C)\leq \textcolor{black}{n}(m-1)!\]
    
    \item [(ii)]  For any $m>2$, there exist polynomials  $f,g$, \color{black}  $\deg f=m$, $\deg g=n$, \color{black}and a cycle  $C$ such that 
    $Z_\Delta(f,g,C)=n(m-1)!$ 

    \item[(iii)] If $m=2$, then $Z_\Delta(f,g,C)=0$, for any polynomial $f$ and $g$ of degrees $m$ and $n$ respectively and a cycle $C$ of $f$. 
    \end{itemize}
    \end{prop}

    \begin{proof} [Proof of Proposition \ref{n<m}]
The proof follows the same general lines as the proof in the tangential case. However, here one considers, on the one hand the \emph{connection curve} $\Gamma_{f+\epsilon g}$ defined as the curve $\Gamma_f$ in the solution of the tangential problem, but with $f+\epsilon g$ playing the role of $f$. 
On the other hand, one can  take either the \emph{zero hypersurface} $S_f$ of $f$ or the zero hypersurface $S_g$ of $g$, see Remark \ref{rem:fg}. 
We work with $Sg$, which is of lower degree $n$.

We have $\deg (f+\epsilon g)=m$. As in the tangential case, one proves that $\deg \Gamma_{f+\epsilon g}=(m-1)!$. The degree of the zero surface $S_g$ is $\deg g=n$, giving by Bezout's theorem,  that the number of intersection points of $\Gamma_{f+\epsilon g}$, with $S_f$ is $n(m-1)!$.
Any such intersection point gives rise to a value $t$ solution of $\Delta(t,\epsilon)=0$. 
This shows (i). 

In order to prove (ii), i.e. the realization of the bound for a suitable $f$, $g$, $C$, note first, that several such points could correspond to the same value of $t$ and the bound would not be sharp. 
Similarly, Bezout's theorem also counts the points at infinity $L_\infty$, which should not be take into account in the problem.

Then, by Remark \ref{gen asym}, there exists a regular at infinity  asymmetric cycle $C_0$ of $f$. By continuity, and the fact that the conditions of being regular at infinity  and asymmetric are open, it follows that the cycle $C_\epsilon$ of $f+\epsilon g$ is still regular at infinity and asymmetric. Due to regularity at infinity,  there are no points of intersection of $\Gamma_{f+\epsilon g}$ with $S_g$ at infinity, for $\epsilon\ne0$.

Next, due to asymmetry, as in Lemma \ref{differentt}, we show that generically,  each point of intersection of $\Gamma_{f+\epsilon g}$ with $S_g$ corresponds to a different value of $t$. This proves (ii).

To prove (iii), for $m=2$, then necessarily the cycle $C_0$ is simple, so Lemma \ref{differentt} does not apply. We assume, without loss of generality, that the cycle $C_\epsilon(t)$ is $C_\epsilon(t)=z_1(t,\epsilon)-z_2(t,\epsilon)$.
If $n=1$, then by Bezout, we get that $\Gamma_{f+\epsilon g}\cap S_g$ consists of one single point. However, the point given by $z_2(t,\epsilon)\equiv z_1(t,\epsilon)$ certainly belongs to this intersection and then $\Delta_{C_\epsilon}\equiv 0$, along this cycle, so this does not correspond to a regular value of $t$, thus giving $Z_\Delta=0$. 
Similarly, if $n=0$, then $\Delta\equiv0$, so there are no regular solutions $t$.
\end{proof}

\begin{prop}\label{n>m}
Let $\deg f=m$, $\deg g=n$.  If {$n\geq m>2$}, then 
\begin{itemize}
    \item[(i)] If $m$ does not divide $n$, then  
for  any cycle $C$, \color{black} 
    \[Z_\Delta(f,g,C)\leq \frac{m(n-1)!}{(n-m)!}.\]
Moreover, the above bound is attained for generic $f$, $g$
 and cycle $C_0$.

    \item[(ii)] 
    If $m$ divides $n$, then
    \[ Z_\Delta(f,g,C) \leq \frac{m(n-1)!}{(n-m)!}-(m-1)!.\]
Moreover, the above bound is attained for generic $f$, $g$
 and cycle $C_0$.

    \item[(iii)] If $m=2$, then $Z_\Delta(f,g,C)\leq \left[\frac{\textcolor{black}{n-1}}{2}\right]$. 
    Moreover, the above bound is attained for generic $f$, $g$
 and cycle $C_0$.
    \end{itemize}
\end{prop}

In order to prove Proposition~\ref{n>m}, we will need some auxiliary results.  Recall that if  $n>m$, then for a deformed cycle $C_\epsilon$ of $f+\epsilon g$, 
 we complete the weights $n_j$ of $C_0$, of the roots $z_j(t,\epsilon)$ of $f+\epsilon g$, not corresponding to roots of $f$, by putting $n_{m+1}=\ldots=n_n=0$.

\color{black}
\begin{lema}\label{lem:sym}
   For $m>2$ and asymmetric cycle $C_0$, the symmetry group $H_\epsilon$, $\epsilon\ne0$, of $C_\epsilon$ has order  $(n-m)!$.
   For $m=2$, the symmetry group of $C_\epsilon$ has order $2(n-\textcolor{black}{2})!$.
\end{lema}
\begin{proof}
For $m>2$, the symmetry group of the cycle corresponds to permutations of the $n-m$ roots of $f+\epsilon g=t$, which go to infinity and do not correspond to any root of $f=t$. The asymmetry of the cycle $C_0$ assures that there are no other permutations in its symmetry group.

For $m=2$, the cycle $C_0$ is necessarily simple. The symmetry group corresponds to permutations of the $n-\textcolor{black}{2}$ roots going to infinity, but also the \textcolor{black}{permutation of the} two roots of $f=t$ \textcolor{black}{belongs to $H_\epsilon$, $\epsilon\ne0$}.
\end{proof}

\begin{rema}
    More generally, for any cycle $C_0$ of $f$  and its deformed cycle $C_\epsilon$ of $f+\epsilon g$, the symmetry group $H_0$ of $C_0$ and $H_\epsilon$ of $C_\epsilon$ are related by 
    \begin{equation*}
        H_\epsilon=H_0\times S_{n-m},
    \end{equation*}
    where $S_{n-m}$ is the permutation group of $(m+1,\ldots,n)$ corresponding to permuting the roots of $f+\epsilon g=t$, which to do not correspond to any root of $f=t$ and come with coefficient $0$ in $C_\epsilon$. 
\end{rema}

\color{black}
The curve  $\Gamma_{f+\epsilon g}\subset\C P^n$ intersects $L_\infty$ at points $p_\alpha=(1:\xi^{\alpha_2}:\dots:\xi^{\alpha_m}:\cdots:\xi^{\alpha_{n}}:0)$, where $\xi=e^{\frac{2\pi i}{n}}$ and $\alpha=({\color{black}n},\alpha_2,\dots,\alpha_{n})\in Sym_n$. 
Among these points we study which belong also to $S_f$ i.e. verify the condition 
\begin{equation}\label{S_f}
 \sum_{j=1}^m n_j({\xi^{\alpha_j}})^m+  \sum_{j={m+1}}^n 0(\xi^{\alpha_j})^m=0.
\end{equation}

We want to count the points $p_\alpha$ belonging to $\Gamma_{f+\epsilon g}\cap S_f\cap L_\infty$. 
We will show that for fixed degrees $m$ and $n$ of $f$ and $g$ some such points are unavoidable. We count this number in function of $m$ and $n$. The others can be avoided for a convenient choice of the cycle $C_0$, i.e. choice of $n_1,\ldots n_m$.

\begin{lema}\label{lem:mdeg} Let $m=\deg f\leq n=\deg g$, 
\begin{itemize}
    \item[(i)] If $m$ does not divide $n$, then the set $\Gamma_{f+\epsilon g}\cap S_f\cap L_\infty$ is empty.

    \item[(ii)] If $m$ divides $n$, then there are at least $(m-1)!(n-m)!$ points in $\Gamma_{f+\epsilon g}\cap S_f\cap L_\infty$.
    Moreover, there exist cycles $C_0$, such that the above bound is exact. \footnote{Note that the hypersurface $S_f$ depends on the cycle $C$.} More precisely, this is the exact bound for cycles belonging to a complement of a finite number of hyperplanes. 
\end{itemize}

\end{lema}

\begin{proof}

Note first that if 
\begin{equation}\label{m-deg}
\xi^{\alpha_j m}=1, 
\end{equation}
for $j=1,\ldots,m,$
then by the cycle condition equation \eqref{S_f} will be satisfied. 

Note that there are $d=\gcd(m,n)$ roots of $z^n=1$ of order dividing $m$, i.e. distinct solutions $\xi^{\alpha_j}$ of \eqref{m-deg}. 
We want to have $m$ distinct solutions of $\eqref{m-deg}$ That is $d=m$ i.e. $m$ divides $n$.
This means that we have $m$ distinct solutions $\xi^{\alpha_j}$, $j=1,\ldots,m$, of \eqref{m-deg}, 
if and only if $m$ divides $n$.
If $m$ divides $n$, we form points $p_\alpha=(1:\xi^{\alpha_2}:\cdots:\xi^{\alpha_m}:\cdots:\xi^{\alpha_n}:0)$ in $\Gamma_{f+\epsilon g}\cap S_f\cap L_\infty$. Permuting the first $m-1$ coordinates or the last $n-m$ coordinates gives a solution. 

Hence, for any cycle $C$, if  $m$ divides $n$, there are $(m-1)!(n-m)!$  points of this type and no points, if $m$ does not divide $n$.
One can avoid having other points, by choosing cycles $C_0$ (i.e. the set $S_f$)  so that conditions \eqref{S_f} is not verified for any $\alpha=(n,\alpha_2,\ldots,\alpha_n)$, which does not correspond to the already found solutions (see \eqref{m-deg}).

\end{proof}
In the case $m=2$, we need a version of the above Lemma \ref{lem:mdeg}, for $m=2$, but with the reduced zero-hypersurface $\tilde S_f$. 

\begin{lema}\label{lem:m=2} Let $m=2=\deg f\leq n=\deg g$, 
\begin{itemize}
    \item[(i)] If $n$ is  odd, then the set $\Gamma_{f+\epsilon g}\cap \tilde{S}_f\cap L_\infty$ is empty. 

    \item[(ii)] If $n$ is even, then there are at least $(n-m)!$ points in $\Gamma_{f+\epsilon g}\cap \tilde{S}_f\cap L_\infty$.
    Moreover, there exist cycles $C_0$, such that the above bound is attained.  More precisely, this is the exact bound for cycles belonging to a complement of a finite number of hyperplanes. 
\end{itemize}

\end{lema}

\begin{proof}
 The proof follows the same general lines as the proof of the previous Lemma. Only at infinity we work with $f(z)=z^2$ and we consider the
 \emph{ reduced} zero hypersurface $\tilde{S}_f$ at infinity. It is simply given by the equation $z_1+z_2=0$. Denote by $\tilde{p}_\alpha=(1:\xi^{\alpha_2}:\cdots:\xi^{\alpha_{n}}:0)$  the points belonging to 
$\Gamma_{f+\epsilon g}\cap \tilde{S_f}\cap L_\infty$. In homogenized coordinates, we put $\xi^{\alpha_1}=1$ and end up with the equation $1+\xi^{\alpha_2}=0$. The result then follows as in the proof of Lemma \ref{lem:mdeg}.
\end{proof}

\begin{proof} [Proof of Proposition \ref{n>m}]

Take $2<m\leq n$. We have to take into account the intersections points at infinity, as well as the fact that each root of the displacement function $\Delta$ corresponds to $k$ $m$-tuples $(z_1,\ldots,z_m)$, where $k$ is the order or the symmetry group $H$ of $C_\epsilon$

\color{black}
Assume moreover that  $m$ does not divide $n$. Then $\deg \Gamma_{f+\epsilon g}=(n-1)!$ and $\deg S_f=m$. By Bezout's theorem,
they intersect in $m (n-1)!$ points.
By Lemma~\ref{lem:sym}, the symmetry group $H$ of $C_\epsilon$ has order $(n-m)!$. 
This proves the bound (i)(a).

Next, by a convenient choice of the cycle $C_0$ (see Lemma \ref{lem:mdeg}), there are no points at infinity that we have to subtract. 

Moreover, for each level set $t$, we obtain $(n-m)!$ points in the intersection, so the number of solutions $t$ is $\frac{m(n-1)!}{(n-m)!}$. Thus showing (i)(b).

Next, if   $m$ divides $n$, then we need to remove the intersections at infinity using Lemma \ref{lem:mdeg}. We obtain
\[
Z_\Delta(f,g,C) \leq \frac{ m(n-1)!-(m-1)!(n-m)!}{(n-m)!}=
\frac{ m(n-1)!}{(n-m)!}-(m-1)!.\] 

The realization of the bound follows from Lemmas \ref{lem:sym} and \ref{lem:mdeg}.

\color{black}

\color{black}
Suppose $m=2\leq n$. Then, in order to prove (ii), one notes that in this case $\deg(f+\epsilon g)=n$, and $\deg \Gamma_{f+\epsilon g}=(n-1)!$. One counts the intersections of the connection curve $\Gamma_{f+\epsilon g}$ with the  zero hypersurface $\tilde S_f$. Note that $\deg  \tilde{S}_f=1$ and the order of the symmetry group $H$ of a simple cycle is $2(n-2)!$. 

We consider first the case $n$ odd i.e. $m$ does not divide $n$. We apply Lemma \ref{lem:m=2}. It says that there are no points of intersection of $\Gamma_{f+\epsilon g}\cap \tilde{S}_g\cap L_\infty$. Hence, we get the bound 
 $Z_\Delta(f,g,C)\leq \frac{(n-1)!}{2(n-2)!}=\frac{n-1}{2}.$

If $n$ is even, then as previously, we have $\deg(f+\epsilon g)=n$, and $\deg \Gamma_{f+\epsilon g}=(n-1)!$, the order of the symmetry group $H$ is $2(n-2)!$. However, now we have $(n-2)!$ points at infinity (by Lemma \ref{lem:m=2}), which we have to subtract before dividing by the order of the symmetry group.
We get $Z_\Delta(f,g,C)\leq \frac{(n-1)!-(n-2)!}{2(n-2)!}=\frac{n-2}{2}$. The above bounds are exact by Lemmas \ref{lem:sym} and \ref{lem:m=2}. These bounds give the common optimal bound and prove (ii).

\end{proof}

The same degeneracies as in the tangential problem force a smaller number of solutions in exceptional cases. But it is possible that the tangential problem has some degeneracy while the infinitesimal problem does not have that degeneracy, so the number of solutions of the infinitesimal problem might be strictly greater than the number of solutions of the tangential problem. This difference gives the alien limit cycles. 

\medskip

{
For instance, the difference of the number of points on $L_\infty$ is considered in Theorem~\ref{theo:C}. Theorem~\ref{theo:C} follows directly by comparison from Theorems~\ref{theo:A} and~\ref{theo:B}.
The remaining case of $m=n$ and a generic $g$ reduces to the case $n=m-1$ by Lemma~\ref{lem:reduction}, giving the same bound $(m-1)!$}

\medskip

Next, we show an example where all the points of $\Gamma_{f}\cap S_g$ lie on the diagonals $\Delta_{ij}$, but this does not occur for the points of
$\Gamma_{f+\epsilon g}\cap S_f$, so there appear alien limit cycles.

\begin{exam}
Consider $f(z)=z^3+z^2$, $g(z)=3z^2+z$, and the cycle $C(t) = z_1(t)+z_2(t)-2z_3(t)$.
Then,
\[\Gamma_f = \{(z_1,z_2,z_3)\colon z_1+z_2 + z_1^2+z_1z_2+z_2^2=0,\ z_1+z_2+z_3+1=0\},\]
\[S_g = \{(z_1,z_2,z_3)\colon z_1 + z_2 - 2 z_3 + 3 (z_1^2 + z_2^2 - 2 z_3^2)=0\}.\]

By Theorem~\ref{theo:A}, the maximum number of zeros of $Z_1(f,g,C)$ is $4$, but a direct computation shows that indeed $Z_1(f,g,C)=0$, as the intersection of $\Gamma_f$ and $S_g$ is contained in the union of the diagonals, $\Delta_{12}\cup \Delta_{23}\cup\Delta_{13}$.


\medskip 

Now, consider the infinitesimal version of the problem.
Compute the sets $\Gamma_{f+\epsilon g}$ and $S_f$, which in this case depend on $\epsilon$, giving
\[
\Gamma_{f+\epsilon g} =\{z_1 + z_2 + z_1^2  + z_1 z_2 + z_2^2 + \epsilon(1 + 3 z_1  + 
   3 z_2 ) = 0, 1 + 3 \epsilon + z_1 + z_2 + z_3 = 0\},
\]
\[
S_f=\{z_1^2 + z_2^2 + z_1^3  + z_2^3 - 2 (z_3^2 + z_3^3)\}.
\]
By a computer algebra system, it can be shown that the intersections of the two varieties consists of four points. 
For $\epsilon\neq 0,1/3,(1\pm \sqrt{3})/6$, the coordinates are different, so they do not correspond to singular points of $f+\epsilon g$. Finally, note that we obtain two solutions of $\int_{C_\epsilon(t)} f$, as there is a symmetry in the coefficients of $z_1,z_2$ of the cycle, which implies that points of the intersection of $\Gamma_{f+\epsilon g}$ and $S_f$ have the same symmetry.

{ As in the tangential problem there are no solutions, it follows that the two solutions of the infinitesimal problem, are alien limit cycles.}
\end{exam}


\section{Tangential problem in the exceptional cases}
In  Theorem~\ref{theo:A} we showed that the upper bound for the number $Z_1$ of solutions of the tangential problem is attained for generic $f$ and for generic cycle $C$. 
In this section we study what happens if some of these conditions are 
violated.

Essentially, there are four degeneracies that force a smaller number of solutions:
\begin{enumerate}
\item Several points in $\Gamma_f \cap S_g$ with the same $t$-value.
\item Points of $\Gamma_f\cap S_g$ on $L_\infty$.
\item Points of $\Gamma_f\cap S_g$ on the diagonals $\Delta_{ij}$.
\item Points of $\Gamma_f\cap S_g$ with multiplicity greater than one.
\end{enumerate}

Next, we consider some cases with these degeneracies.

\subsection{Case of more than one point in $\Gamma_f \cap S_g$ corresponding to the same generic value of $t$.}

To each cycle $C(t)=\sum_{i=1}^m n_i z_i(t)$, we associate the tuple $c=(n_1,\ldots,n_m)$, of weights. 
Genericity of a cycle $C$ is expressed in terms of the genericity of the tuple $c$.

Using Vieta mapping  $V$ given in \eqref{Vieta}
one can formulate the  counting of the number of zeros  of the first Melnikov function $M_1(t)$,
as the number of intersections of $V( S_g)$ with the line 
$L_f=\pi_{m-1}^{-1}({\sigma^{m-1}})$.

For generic $ g,C$, the Vieta map $V:S_g\to V(S_g)$ is generically one-to-one, and therefore, for a generic $\sigma^{m-1}$, the points of intersection $L_f$ with $V( S_g)$ are in one-to-one correspondence with the points of intersection of $\Gamma_f$ with $S_g$.

However, if the map $V:S_g\to V(S_g)$ is generically $k$-to one, then   for generic $f$ we have 
\begin{equation*}
   \# (\Gamma_f\cap S_g)=k\cdot\#(L_f\cap V( S_g)).
\end{equation*}
We compute this $k$ for generic $g$ in terms of symmetries of the tuple $c=(n_1,\dots,n_m)$:

\begin{lema}
    Let $H\subset Sym_m$ be the subgroup of $Sym_m$ preserving the tuple $c$ up to sign.
    Then $k=|H|$. 
\end{lema}

\begin{proof}
    For a generic point $\sigma\in \C^m$ the set $V^{-1}(\sigma)$ consists of  $m!$ points $p_i$ with pairwise different coordinates which are roots of the polynomial $f_\sigma(z)$. The group $Sym_m$ acts freely and transitively on $V^{-1}(\sigma)$ by permutations of coordinates. By choosing a point $p_1\in V^{-1}(\sigma)$ we can write $V^{-1}(\sigma)=Sym_m(p_1)$.

    Assume now that $\sigma\in V(S_g)$ is generic and let $\{p_1,...p_k\}=V^{-1}(\sigma)\cap S_g$ be the $k$ points   of $S_g$ sent to $\sigma$ by $V$. 
    Let 
    $$
    H=H_{\sigma, p_1}=\{\alpha\in Sym_m  \,| \, \alpha(p_1)\in S_g\}\subset Sym_m
    $$ be the set of permutations corresponding to $V^{-1}(\sigma)\cap S_g$. Clearly, the set $H$ depends continuously on the choice of $\sigma, p_1$, so is  locally constant near $p_1$.
Therefore, the group $H$ depends on the irreducible component of $S_g$ containing $p_1$ only.

    The group $Sym_m$ acts on $\C^m$. 
    By the above definition, for any $\alpha\in H$ the surface $\alpha(S_g)$ coincides with $S_g$
    locally near $\alpha(p_1)$. If $C$ is not a simple cycle, then $S_g$ is irreducible, which implies that $\alpha (S_g)=S_g$. If $C$ is a simple cycle, $C=z_i-z_j$, then $S_g$ is a union of a hypersurface $\tilde{S}_g$ and the hyperplane $\Delta_{ij}$, and similarly $\alpha(\tilde{S}_g)=\tilde{S}_g$, $\alpha(\Delta_{ij})=\Delta_{ij}$, see Theorem~\ref{th:simpletang}.

     Thus, $H$ is the subgroup of the group $Stab_{S_g}\subset Sym_m$ preserving $S_g$. 
    The opposite inclusion is trivially true: if $\alpha\in Stab_{S_g}$, then clearly $\alpha(p_1)\in Sym_m(p_1)\cap S_g$, i.e. $\alpha\in H$. Thus, $H=Stab_{S_g}$.

    Now, the action of $Sym_m$ on $S_g=\{\sum n_{j}g(z_j)=0\}$ reduces to the action of $Sym_m$ on the tuple $c$ by permutation of coordinates: if $\alpha=(\alpha_1,\dots,\alpha_m)\in Sym_m$ then 
    $\alpha(S_g)=\{\sum n_{\alpha_j}g(z_j)=0\}$. The equality $\alpha(S_g)=S_g$ means that their defining equations are proportional, which happens  if and only if the tuples $c$ and $\alpha(c)=\left(n_{\alpha_1},\dots,n_{\alpha_m}\right)$  coincide up to  sign.
\end{proof}
\begin{exam}
    For $c=(1,-1,1,-1, 0,...,0)$, the subgroup $H$ is generated by permutations of the last $m-4$ entries, and by the four permutations $(1)(2)(3)(4)$, $(13)(24)$, $(12)(34)$ and $(14)(23)$. Thus $|H|=4(m-4)!$

On the other side, for such $c$ and generic $f,g$ we have $\# (\Gamma_f\cap S_g)\le n(m-1)!$, so $Z_1(f,g,C)\le \frac{n(m-1)(m-2)(m-3)}{4}$.
    
\end{exam}

\subsection{Simple cycles}
Several exceptional phenomena occur in the case of simple cycles. Here we calculate the number of zeros of abelian integrals $Z_1(m,n,S)$ along simple cycles, for $\deg f=m$, $\deg g=n$.

\begin{lema}\label{lem:simple}
  For generic $f,g$   and a simple cycle $C$ the number   of points in  $\Gamma_f\cap S_g$ lying outside $\cup\Delta_{ij}$ is  at most $(n-1)(m-1)!$

\end{lema}
\begin{proof}
     For generic $f,g$ the intersection $\Gamma_f\cap \tilde{S}_g$, where $\tilde{S}_g$ is defined in Lemma~\ref{lem:irreducibility}, is disjoint from $\cup\Delta_{ij}$. Thus the number of points in $\Gamma_f\cap \tilde{S}_g$ counted with multiplicities  and outside of $\cup\Delta_{ij}$ is given by the Bezout bound $(n-1)(m-1)!$
\end{proof}
\begin{lema}\label{lem:simple infinity}
Let $f$ be generic and $C$ a simple cycle of $f$.
    \begin{itemize}
        \item [(i)] For $m,n$ coprime the intersection $\Gamma_f\cap \tilde{S}_g$ is disjoint from $L_\infty$.
        \item [(ii)] Denote $\gcd(m,n)=d$. Then $\Gamma_f\cap \tilde{S}_g\cap L_\infty$ consists of $(d-1)(m-2)!$ points.
    \end{itemize}
\end{lema}
\begin{proof}
     We assume that $c=(1,-1,0,\dots,0)$ (the remaining cases are the same up to change of notations). 
     
     (i) The regular at infinity condition \eqref{eq:generic} reads 
    $1-\xi^{n\alpha_1}\not=0$. If $\gcd(m,n)=1$, then $n\alpha_1$ is not divisible by $m$, so the condition is automatically satisfied, which proves the first claim (as in Lemma \ref{noinf}).

   (ii)  If $\gcd(m,n)=d$, then $\xi^{n\alpha_1}=1$, for $\alpha_1=km/d$, $k=1,\dots,d-1$. To each such $k$, correspond $(m-2)!$ points of $\Gamma_f\cap L_\infty$, with the same first two coordinates $1, \xi^{\alpha_1}$, thus also lying on $\tilde{S}_g$. Thus $\Gamma_f\cap \tilde{S}_g$ intersects $L_\infty$ at $(d-1)(m-2)!$ points.
\end{proof}

\begin{theo}\label{th:simpletang}
 Let $Z_1^S(m,n)$ denote the maximal number of zeros of abelian integrals along \emph{simple} cycles of $f$ of degree $m$ deformed by $g$ of degree $n$ (Remark \ref{remCheb}). Then 
 $$
 Z_1^S(m,n)=\frac{(n-1)(m-1)-(d-1)}{2}.
 $$
\end{theo}

\begin{proof}
   A simple cycle $C$ is encoded by a tuple $c=(1,-1,0,...,0)$. The symmetry group $H$ of the cycle $C$ is generated by permutations of the last $m-2$ entries  and permutation of the first two entries of $c$. Thus $|H|=2(m-2)!$
   
   On the other side, for such $C$ and generic $f,g$ the number of points of $ \Gamma_f\cap S_g$ lying outside $\cup\Delta_{ij}$ is at most $(n-1)(m-1)!$ (see Lemma~\ref{lem:simple}). The group $H$ acts freely on these points. Thus, the number of points of $L_f\cap V( S_g)$, which are not critical points of $f$ is at most $\frac{(n-1)(m-1)!}{|H|}=\frac{(n-1)(m-1)}{2}$. This is the bound given by  Gavrilov and Movasati in \cite{GM}. 

   Lemma~\ref{lem:simple infinity} implies that this upper bound is sharp if $\gcd(m,n)=1$.
   However, if $\gcd(m,n)=d>1$, then at least $(d-1)(m-2)!$ of the points of $ \Gamma_f\cap S_g$ lie on $L_\infty$, see Lemma~\ref{lem:simple infinity}. Thus, in the general case, $Z_1^S(m,n)$ is lower, given by
   \begin{equation*}
       \frac{(n-1)(m-1)!-(d-1)(m-2)!}{2(m-2)!}=\frac{(n-1)(m-1)-(d-1)}{2}.
   \end{equation*}
   The bound is sharp, as Bezout theorem gives the exact bound.
\end{proof}
   



\subsection{Points of $\Gamma_f\cap S_g$ on $L_\infty$}
We showed previously that $\#(\Gamma_f\cap L_\infty)=(m-1)!$ We study now which of these points can be in $S_g$, as well. 
This can happen for some non-generic tuples $c=(n_1,\ldots,n_m)$ depending on the arithmetic properties of $m,n$. 

Recall that  the degree of  \emph{cyclotomic field extension} $\Q(\xi)$, where $\xi$ is a primitive $m$-th root of $1$, is given by $[\Q(\xi):\Q]=\phi(m)$, where $\phi$ is the \emph{Euler totien function}, see \cite{Lang}.
\begin{lema}\label{lem:infinity}
\hfill
  \begin{itemize}
     \item[(i)] 
     $\#\left(\Gamma_f\cap S_g\cap L_\infty\right)<(m-1)!$ for any nontrivial cycle $C$.
     \item[(ii)] If $m$ is prime, then $\Gamma_f\cap S_g\cap L_\infty=\emptyset$, for any nontrivial cycle $C$.
      \item[(iii)] If $m$ and $n$ are coprime, then the set of tuples $\{n_i\}$ such that $(1:\xi:...:\xi^{m-1})\in\Gamma_f\cap S_g\cap L_\infty$ is a free $\Z$ module generated by $m-\phi(m)$ linearly independent tuples.
  \end{itemize}  
\end{lema}

\begin{proof}

We claim that $\left\langle p_\alpha=(\xi^{\alpha_1},\dots,\xi^{\alpha_m}), \alpha\in Sym_m\right\rangle^\perp=\langle(1,\dots,1)\rangle$. Indeed, suppose that $\ell=\sum n_jz_j$ vanishes on all $p_\alpha$, and let us prove that $n_1 =n_2$. Evaluating on $p_{id}$ and $p_{(12)}$ we get 
$\ell(p_{id})=\sum n_j\xi^j=0$ and also $\ell(p_{(12)})=n_1\xi^2+n_2\xi+\sum_{j>2} n_j\xi^j=0$. Subtracting, we get $(n_1-n_2)(\xi-\xi^2)=0$, so $n_1=n_2$.

Thus, the  linear functionals vanishing on all vectors $p_\alpha$ are all  proportional to $\sum z_j$. However, $C$ is a cycle, i.e. $\sum n_j=0$. Hence,  $C$ cannot be proportional to $\sum z_j$ unless $C=0$. This proves (i).

(ii) If $m$ is prime, then $[\Q(\xi):\Q]=\phi(m)=m-1. $
Hence, there is only one integer relation among the roots $\xi^j$, namely the aforementioned  relation $\sum_{j=1}^m\xi^j=0$. But this relation corresponds to the chain $C=\sum z_i$, i.e. to the tuple $c=(1,\ldots,1)$, which does not correspond to a cycle. 

(iii) As $m$ and $n$ are coprime, the set $\{\xi^{nj}\}_{j=1}^m$ coincides with the set of roots of unity of degree $m$. 
As $[\Q[\xi]:\Q]=\phi(m)$, the space of linear relations over $\Q$ between these roots has dimension $m-\phi(m)$.
\end{proof}

\begin{exam}
For $m=pq$, with $p,q$ prime, the group $\{\xi^j, j=1,...,m\}$ is isomorphic to $\Z_p\times \Z_q$, with isomorphism sending  the primitive roots of unity $e_p=\xi^q$ and $e_q=\xi^p$ of degree $p$ and $q$, resp., to the generators of $\Z_p$ and $\Z_q$. We have $p$ relations of the form $$e_p^j\sum_{k=0}^{q-1} e_q^k=0 , \quad j=0,..,p-1,$$
and $q$ relations of the form 
$$e_q^k\sum_{j=0}^{p-1} e_p^j=0,\quad k=0,..,q-1.$$ 
The first relation means that sum of any row in the matrix $\{e_p^je_q^k\}_{j,k=0}^{p-1,q-1}$ is zero, the second relation means that the sum of any column is zero.
There is one linear dependence between these relations due to double counting:  the sum $\sum \xi^j=0$ of all elements of this matrix can be computed as sum of all rows or sum of all columns.

Altogether we have $p+q-1=pq-\phi(pq)$ linearly independent relations, as needed.
\end{exam}
\begin{exam}
    For $m=p^k$ we similarly have $p^{k-1}=p^k-\phi(p^k)$ integer relations $\xi^j\sum_{i=0}^{p-1}\xi^{p^{k-1}i}=0 $, with $j=0, p^{k-1}-1$.
\end{exam}
\begin{rema}
    Combining the last two examples, one can get the full description of the generators of the module of integer linear relations between $\{\xi^j\}$, for any $m$.
\end{rema}

If $\gcd(m,n)=\ell>1$, then the numbers $\xi^{m\alpha_i}$ run over all roots of unity of degree $n/\ell$  and we have a similar  situation but with smaller degree.

\subsection{Tangential problem for any cycle $C$}
We proved in Lemma~\ref{trans} that for a non-simple cycle $C$, any polynomial $g$ and a generic polynomial $f$ the intersection of $\Gamma_f\cap S_g$ is transversal outside of $L_\infty$. This claim doesn't exclude the possibility that this intersection is empty (i.e. $\Gamma_f\cap S_g\subset L_\infty$).

In Theorem~\ref{theo:A} we counted the number of points in  this intersection for generic $f,g$ and $C$, in particularly we showed that it is not empty. The examples above show that in exceptional cases the number of points can be smaller.
Here we prove that this intersection is not empty for \emph{any} $g,C$ and generic $f$. 

\begin{prop}
For  any polynomial $g$, any cycle $C$ and a generic  polynomial  $f$ of degree $\deg f>2$  
 the set  $\left(\Gamma_f\cap S_g\right)\setminus L_{\infty}$ is non-empty.
\end{prop}

\begin{proof}
We use the Vieta mapping defined by \eqref{Vieta}.
    The claim is equivalent to the claim that the map $\pi_{m-1}\circ V: S_g\to \C^{m-1}$ is dominant, i.e. that the set $B=\pi_{m-1}(V(S_g))$ is Zariski dense. Assume that it is not,  i.e. that $B$ has dimension less than $m-1$. As $V$ is finite, we have $\dim V(S_g)=m-1$. As $\pi_{m-1}$ is a linear projection with one-dimensional kernel, the above assumption means that  $\dim B=m-2$ and $V(S_g)=B\times \C_t$. 
The latter condition means that,  if $\int_Cg=0$, for some $f=f_{\sigma^{m-1}}$, then $\int_Cg=0$, for  $f-t$, for any $t\in\C$, as well, i.e. $B=\{f\,| \,\int_{C(t)}g\equiv0\}$. In other words, if $\int_Cg$ vanishes for some $f$, then it vanishes  identically.
\begin{lema}\label{lem:composition}
    Assume that $f\in B$ and the space of $0$-cycles of $f$ has no subspace invariant under monodromy. Then     $g=P(f)$ for some $P\in \C[t]$, $P(0)=0$.
\end{lema}
\begin{proof}
    Clearly, this implies that $\int_{C'}g=0$, for \emph{any} cycle $C'$, in particular, for any simple cycle. This implies that $g$ takes the same value $P(t)$ at all roots of $f(z_i)-t$, for any $t$. Standard arguments show that $P$ is necessarily a  polynomial.
\end{proof}
If, for a  generic $\sigma_{m-1}\in B$, the polynomial $f_{\sigma^{m-1}}$ satisfies the conditions of Lemma~\ref{lem:composition}, e.g if $f$ is Morse polynomial, then we arrive to a contradiction: as  $m-2>0$,  we have a continuous  family of polynomials $f_{\epsilon}$ vanishing at $0$ such that $g=P_\epsilon(f_\epsilon)$, i.e. all $f_\epsilon$ have the same level sets, in particular their roots coincide, which is impossible. 

Thus we remain with the case when the codimension-one set $B$ consists of non-Morse polynomials. We prove that this is impossible in this case, as well. There are two cases corresponding to two codimension-one  strata of the set of non-Morse polynomials:
\begin{enumerate}
    \item  the strata of polynomials with  one  triple point, and all remaining critical points   simple and with all critical values  distinct. Picard-Lefschetz formula implies that the linear space $Mon(C)$ spanned by the orbit of $C$ under monodromy contains an integer  cycle vanishing at the triple point. The classical monodromy of the triple point acts on the two-dimensional space of vanishing cycles by rotation, and both eigenvectors are non-real.  Thus $Mon(C)$ contains all cycles vanishing at the triple point. By Picard-Lefschetz, it then contains all other simple cycles and we conclude as above.
    \item the strata of polynomials with only simple critical points and all critical values distinct except exactly two. If $m>4$ then the monodromy is transitive and Lemma~\ref{lem:composition} applies. 
    
    For $m=4$, then these polynomials have the form $f(z)=\left( (z-a)^2-b\right)^2-c$, $b\not=0$, where $c=(a^2-b)^2$, which are all topologically the same: they are conjugated to $z^4-2z^2$ by left and right action of $\Aff_1$. 
    Let  $f(z)=z^4-2z^2$ and let $z_1<z_2<z_3<z_4$, $z_i=z_i(t)$,  be the roots of $f=t, t\in(-1, 0)$. Then the space $L_2$ generated by the cycles $z_1-z_4$ and $z_2-z_3$ is invariant under monodromy, and (as monodromy acts by orthogonal transformations) its orthogonal complement $L_1=L_2^\perp$ generated by $z_1-z_2-z_3+z_4$ is invariant under monodromy, as well. These are the only two subspaces invariant under monodromy.

   \begin{itemize}
       \item If $\int_Cg\equiv 0$ and $C\notin L_1, L_2$, then $Mon(C)$ spans the whole space of cycles and we conclude as above.
       \item If $C\in L_2$ then we have $g(z_1)\equiv g(z_4)$, $g(z_2)\equiv g(z_3)$, i.e. $g$ is an even polynomial, $g(z)=g(-z)$. 
       Repeating the same arguments for the shifted polynomial $\tilde{f}=f(x-1)$ also lying in $B$, we conclude that necessarily $g(z)=g(-z+2)$ as well, which is clearly impossible.
       \item Assume that  $C=z_1-z_2-z_3+z_4$ and fix some $t$. We have $g(z_1)-g(z_2)=g(z_3)-g(z_4)$. By symmetry, the shifted polynomial $f_1(z)=f(z-\alpha)\in B $ with $\alpha=z_3-z_1$ satisfies  $f_1(z_3)=f_1(z_4)=t$. Let $z_5=z_3+\alpha< z_6=z_4+\alpha$ be the two other roots of $f_1(z)-t=0$. As  $f_1\in B $, we have $g(z_1)-g(z_2)=g(z_5)-g(z_6)$ as well. Repeating this argument, we get $g(z_1)-g(z_2)=g(z_3+k\alpha)-g(z_4+k\alpha)$, for any $k\in \Z$, which is clearly impossible.
   \end{itemize}

\end{enumerate}
\end{proof}
The previous proof implies that for any $g$, $C$ the set of $f$ for which $\int_{C(t)}g\equiv0$ has codimension at least two. Results of \cite{GP}, \cite{ABM} show that this set is empty for a generic $g$.



\section{Dimension of the space of abelian integrals \\
and Chebyshev property conjecture}

In this section we study the dimension $\dim({f,n,C})$ of the space of abelian integrals along zero-cycles $C$ of a polynomial $f$, corresponding to deformations $g$ of degree at most $n$. Note that this dimension gives a lower bound for the number of zeros $Z_1(f,n,C)$ of the corresponding abelian integrals (tangential problem) and the number $Z_\Delta(f,n,C)$ of $0$-limit cycles (i.e. solutions of the infinitesimal problem):

\begin{equation*}
   \dim(f,n,C)-1\leq Z_1(f,n,C)\leq Z_\Delta(f,n,C).
\end{equation*}
Indeed, from the independence of elements of a basis $I_0,\ldots,I_{\dim(f,n,C)-1}$ of the space of abelian integrals, it follows that, except for a finite set $S$ of values $t$, all the Wronskians $W(I_0,\ldots,I_{k})$, $k=0,\ldots,\dim(f,n,C)-1$, of the abelian integrals of the basis are non-zero. Hence, by Cramer's rule, 
 it follows that  for any values $t_1,\ldots,t_{\dim(f,n,C)-1}$ belonging to an interval in a complement of $S$, there exists a polynomial $g$ of degree less then or equal to $m$, such that the abelian integrals verify $\int_{C(t_i)}g=0$, for $i=1,\ldots,\dim(f,n,C)-1$. See also 
the theory of Chebyshev systems (for instance \cite{M}).
The function $g$ is given as a solution of a  system of linear equations. This system is regular by the Chebyshev property. The second inequality follows from the implicit function theorem, as any regular solution of the tangential center problem gives a solution of the infinitesimal center problem. Recall that the converse is in general false as shown by the examples of \emph{alien limit cycles}.

Varying  the above polynomial  $f$, with $f$ of degree less then or equal to $m$ and  its cycle $C$  and taking the maximum, one gets the analogous inequality
\begin{equation*}
    d-1=\dim(m,n)-1\leq Z_1(m,n)\leq Z_\Delta(m,n).
\end{equation*}
\medskip

 \begin{rema} 

One can extrapolate Arnold's Chebyshev property conjecture to abelian integrals on cycles of any dimension $k$. Here, in particular, on abelian integrals on $0$-cycles. 

\end{rema}

This Chebyshev property of abelian integrals  on $1$-cycles appearing for generic quadratic deformations of exact Hamiltonian systems was proved by putting together the result \cite{G} of Gavrilov and  \cite{ZL} of Zenghua Zhang and  Chengzhi Li.
Here the deformations are given by a families of polynomial one-forms deforming an exact form and the degree is the maximum of the degrees of their coefficients. Many examples where the Chebyshev property is verified on 1-cycles are given by Petrov and other authors.
In \cite{RZ} some  symmetric integrable systems deformed by symmetric forms are studied and the authors  prove that the corresponding abelien integrals do not verify the  Chebyshev property. These examples are nongeneric.  

However, abelian integrals along cycles $C$ of a generic polynomial $f$ of degree $m$ of polynomials $g$ of degree less then or equal to $m-1$, should certainly be considered as natural abelian integrals.

\begin{rema} 
Note, however, that 
Theorem~\ref{theo:D} shows that 
abelian integrals along zero-dimensional cycles 
for $\deg f=n$ and $\deg g=n-1$ (or $\deg g=n$) do
not form Chebyshev systems. Worse, they are very far from being Chebyshev. 

\end{rema}

\subsection{Brieskorn and Petrov moduli}
Let $f\in\C[z]$ be of degree $m$. 
Following \cite{GM}, we define the \emph{Brieskorn modulus} $\mathcal{B}$ of $f$ over $\C[t]$  by

\begin{equation*}
\mathcal{B}=\frac{\C[z,t]}{(f(z)-t)\C[z,t]+\C[t]}.
\end{equation*}
This is the space of restrictions to the graph of $f$ of polynomials in $z,t$ identically vanishing on $\{z=0\}$.
In other words, this  is an algebraic extension of $\C[t]$ by an element $z$ satisfying the equation $f(z)=t$. 

Let now $C$ be a $0$-cycle of $f$. We define a mapping
$$
\mathcal{I}_{C}:\mathcal{B}\to \mathcal{A},
$$
where $\mathcal{A}$ is the space of algebraic functions in one variable $t\in\C$, 
given by 
$$
\mathcal{I}_C(g)(t)=\int_{C(t)}g.
$$
Note first that the function $\mathcal{I}_C$
is well defined. We call its image  the \emph{Petrov modulus} $\mathcal{P}_{C}$ over $\C[t]$.
Note that the Petrov modulus depends on the cycle $C$.
Given a polynomial $g(z,t)$ its image in the Petrov modulus is the same as the image of $g(z,f(z))$.

\begin{rema}
    The ring $\C[z]$ is a free $\C[t]$-module generated by $1,\dots, z^{m-1}$. In other words, any polynomial $g\in\C[z]$ can be represented as a sum
    $$
    g=P_0(f)+zP_1(f)+...+z^{m-1}P_{m-1}(f), \quad P_i\in\C[t],$$ 
    and the above proof shows that
    $$
    \int_Cg=P_1(t)\int_Cz+...+P_{m-1}(t)\int_Cz^{m-1}=\int_C\tilde{g}, $$
where $\tilde{g}=zP_1(f)+...+z^{m-1}P_{m-1}(f)$ has degree not divisible by $m$. See the following subsection for more details.
\end{rema}

Recall that by \cite{GM}, the Brieskorn modulus $\mathcal{B}$
is a free $(m-1)$-dimensional modulus generated by the function $z^j$, $j=1,\ldots,m-1$. 

The proof follows easily by the Euclidian division algorithm by $f$ recalling that $\mathcal{I_C}(1)=0.$

Denote $\mathcal{M}$ the monodromy group of $f$. Consider its action on the $\Z$-modulus of cycles $ \mathcal{C}.$ Let $Orb(C)$ be the image of $C$ by the monodromy group in the modulus of cycles of $f$. 

\begin{prop}
If the orbit by monodromy of a cycle $C$ is the whole modulus of cycles $\mathcal{C}$, then the Petrov modulus $\mathcal{P_C}$ and the Brieskorn modulus $\mathcal{B}$ are isomorphic. 
\end{prop}

\begin{proof}
If the orbit $Orb(C)$ is the whole  cycles modulus $\mathcal{C}$, then in particular it contains all the simple cycles. Then the claim follows from Proposition 3 in \cite{GM}.

Let $c=(c_1,...,c_{m})\in\Z^m$ be the vector of coefficients of $C$, and denote $\phi (x)=c_1+c_{m}x+...+c_2x^{m-1}$.
Let $\ell_c=\deg \gcd(\phi(x), x^m-1)$.

\begin{prop}\hfill
   \begin{itemize}
       \item [(i)] The orbit of any  cycle generates a submodule of the modulus of cycles $\mathcal{C}$ of dimension at least $m-\ell_c$.
       \item [(ii)] In particular, for $m$ prime the orbit of any cycle generates the whole  $\mathcal{C}$.
   \end{itemize}  
\end{prop}
\begin{proof}
The monodromy of $f$ at infinity is a cyclic permutation $\alpha=(1,...,m)$ of roots of $f$. Denoting by $\alpha^*$ the action of $\alpha$ on cycles, then we see that the vectors $c, \alpha^*(c),...,(\alpha^*)^{m-1}(c)$ form a \emph{circulant matrix} $A^*(C)$, whose rank is $m-\ell_c${\color{black}, (see \cite{Ing})}. Thich proves the first claim.

More exact, the eigenvalues of $A^*(C)$ are $\lambda_j=\phi(\xi^j)$, $j=0,...,m-1$. Clearly, $\lambda_0=\sum c_i=0$.

Recall (see  Lemma~\ref{lem:infinity}) that, for prime $m$, the only  integer relation between $\xi^j$ is $\sum \xi^j=0$.
This implies that $\lambda_j\not=0$, for all $j\not=0$. Thus $rk A^*(c)=m-1$ equals the dimension of the space of cycles, i.e. $\{(\alpha^*)^j(c)\}_{j=0}^{m-1}$ generate the space of cycles, which proves the second claim.
\end{proof}

For compositions of polynomials the situation can be more complicated.
\begin{exam}
    For $f=(x^3-1)^6$ the critical points are $ 1, \sqrt[3]{1}, (\sqrt[3]{1})^2,0$, with critical values $0,1$ correspondingly.
    The monodromy is generated, by two permutations $(1\,..\,6)(7\,..\,12)(13\,..\,18)$ and $(1 \,7 \,13)$, respectively. A cycle $C=\sum_{j=1}^9\left( z_{2j}-z_{2j+1}\right)$ remains invariant under monodromy.
\end{exam}

\end{proof}

\subsection{Dimension of the Petrov modulus $\mathcal{P_C}$}

Given any $n\in\N$ denote by $\C_n[z]$ the space of polynomials of degree less then or equal to $n$ and 
\begin{equation*}
\mathcal{B}_n=\frac{\C_n[z]}{(f(z)-t)\C[z,t]+\C[t]}.
\end{equation*}

\begin{rema} Note that if the polynomial $f$ is $2$-transitive and the cycle $C$ is simple, then the orbit by monodromy $\mathcal{M}$ of the cycle $C$  generates the whole space of cycles. Hence, in that case the Petrov modulus $\mathcal{P}_C$ is isomorphic to the Brieskorn modulus $\mathcal{B}_n$. Note that in any case $\dim\mathcal{P}_C\leq \dim \mathcal{B}_n$. 

Note however that as shown in Example \ref{ex:z^4-z^2}
the generation of the space of cycles by the orbit of a cycle can depend on the cycle. 

It can be fulfilled for one cycle $C_1$ but not
for another cycle $C_2$.
\end{rema}
\begin{exam}\label{ex:z^4-z^2}
Let $f(z)=z^4-z^2$. For $-1/4<t<0$, there are four real roots of $f$, given by $z_1(t)<z_2(t)<z_3(t)<z_4(t).$
Then, by monodromy, the cycle $C_1=z_1-z_2$ generates the whole space of cycles, whereas the cycle $C_2=z_2-z_3$ does not. 
\end{exam}

We recall from Gavrilov and Movasati \cite{GM} the dimension of the Brieskorn modulus $\mathcal{B}_n$ as a $\C$-vector space.

\begin{prop}\cite{GM}\label{prop:brieskorn}
$$\dim_{\C}(\mathcal{B}_n)=n-\left[\frac{n}{m}\right].$$
\end{prop}

Here, $[x]$, for $x\in\mathbb{R}$, denotes the \emph{integer part} of $x$.

\begin{proof}

Let $g$ be a polynomial of degree $n$ and let $\ell=\left[\frac{n}{m}\right]$.
We first perform Euclidean division of $g$ by $f^\ell$. Next, we divide the remainder by $f^{\ell-1}$ etc. 
We thus get a unique presentation
$$
g(z)= a_\ell(z)f(z)^\ell+\ldots+a_1(z)f+a_0(z),
$$
with $\deg(a_i)\leq m-1$, 
$i=0,\ldots,\ell$ and $\deg(a_\ell)=r= n-\ell m$.

$$
g(z)= a_\ell(z)f(z)^\ell+\ldots+a_1(z)f+a_0(z),
$$
The generators of the Brieskorn modulus will be:
$z,z^2,\ldots,z^{m-1}$, $fz, fz^2,\dots,fz^{m-1},\dots,$
$f^{\ell-1}z,\ldots f^{\ell-1}z^{m-1}$, $f^{\ell}z,\ldots,f^{\ell}z^r$.
They are all independent and their dimension is 
$$
\ell(m-1)+r=n-\ell=n-\left[\frac{n}{m}\right].
$$

\end{proof}

As in \cite{GM} for simple cycles, in general from the $\C$-dimension follows a lower bound for the number of zeros. 

\begin{coro} 
Given a polynomial $f$ of degree $m$ and its cycle $C$, for any $n$, there exists a polynomial $g$ of degree less then or equal to $n$, such that the abelian integral of $g$ along $C$ has at least 

\begin{equation*}
   n-\left[\frac{n}{m}\right]-1.
\end{equation*}
isolated zeros. 
\end{coro}

\section{Concluding remarks and perspectives}

In the present paper we have completely solved the infinitesimal and tangential zero-dimensional versions of 16-th Hilbert problem. Moreover, we have shown the existence of alien limit cycles in this context and provide the mechanisms for their generation. Nevertheless, there are some questions and open problems that have not been solved.

\medskip

The first question that concerns the higher order terms of the displacement function. 
The terms $M_\mu(t)$ in the expansion $\Delta(t,\epsilon)=\sum_{j=1}^\infty M_j(t)\epsilon^j$ of a displacement function $\Delta$ given by \eqref{eq:Delta} are called Melnikov functions.
Generically, $M_1$ is nonzero and is the subject of investigation in our tangential problem. If $M_1\equiv0$,
the first nonzero $M_\mu$ is particularly important (see \cite{F} and \cite{Giter}). It is easy to see that in the study of zero cycles, it is always an abelian integral in the slightly generalized sense:
$$
M_\mu(t)=\int_{C(t)}\frac{P}{(f')^\mu},$$
with polynomial $P$. 

\begin{problem}  
A natural problem is to determine the structure  of the space of first nonzero Melnikov functions for given degrees and ask for the number of their zeros. 

In particular, study the example $f=z^6$, $C(\epsilon)=z_0+2z_1+z_2-z_3-2z_4-z_5$
with $g=z^2+z^3$. Note that in that example $M_1=0$ and $M_2\not=0$.
\end{problem}

Another question that arises is the study of the general nature of the bifurcation diagram of the displacement function $\Delta$ or its tangential part $M_1$ in the space of polynomials $g$. It is an algebraic set if we work in the complex space and a semi-algebraic set if we work in the real space.

\begin{problem}  
Can one describe the bifurcation diagrams for low degrees $n,m$? 
\end{problem}  

Bautin's problem: Consider a deformation $f+\epsilon g$ as in \eqref{eq:var}  of degree at most $m$
with $C(t)$ a cycle of $f$ and $t_0$ a critical value of $f$ for which $C(t)$ degenerates.

\begin{problem}  
Bound, in function of $m$, the number of regular  solutions of the equation $\Delta(t,\epsilon)=0$ in a neighborhood of $t_0$.
\item  What is the maximal multiplicity of cycles depending on the degrees $m$ and $n$.
\end{problem}  

A related problem concerns bounding the \emph{Bautin index}. Given a polynomial $f$, its cycle $C(t)$ and a deformation $g$, consider the displacement function $\Delta(t,\epsilon)=\sum_{j=1}^\infty M_j(t)\epsilon^j$. 
The condition of having a center (i.e. not breaking the family of cycles $C(t)$) is that $\Delta(t,\epsilon)\equiv0$ i.e. $M_1=M_2=\ldots=M_i=\cdots=0$. The coefficients $M_i$ depend on $g$ and by Notherianity for fixed degree $m$ of $g$,
the sequence of ideals $(M_1)\subset(M_1,M_2)\subset\cdots$ stabilizes for some $(M_1,\ldots,M_{b(m)})$.
We call this index $b(m)$ the \emph{Bautin index} of $m$.

\begin{problem} Taking $f$ and $g$ of degree at most $m$, for any cycle $C(t)$, bound the Bautin index $b(m)$.
\end{problem}

In our tangential and infinitesimal problem we ask for a bound of the number of regular zeros of $M_1(t)$ or $\Delta(\epsilon,t)$. These functions are multivalued, but having a finite number of determinations. We study the number of their zeros taking into account all the determinations. We show in Theorem~\ref{theo:D}
that this number is far from the dimension of the space of these functions and hence is far from being Chebyshev.

\begin{problem}  
One can ask however, what happens if one restricts everything in a real domain i.e. for $f$ and $g$
real polynomials and for $t$ belonging to an interval between two adjacent critical values of $f$. What is the bound in that case and can one obtain the Chebyshev property taking only such real  zeros of $M_1$ or $\Delta(t,\epsilon)$.
\end{problem}

Many of the above  notions make sense if one considers chains $C$ instead of cycles i.e., if one does not assume that $\sum_{i=1}^m n_i=0$. The Definition \ref{defint} of the displacement function $\Delta$ and abelian integrals apply. However, the expression \eqref{eq:con_g} expressing the first order term of the displacement function $\Delta$ as an abelian integral of $g$ along the undeformed cycle is no longer valid. It seems to us that it is natural to restrict the study to cycles and not more general chains, but some results can be generalized to chains.

\begin{problem}  
Generalize previous results to chains.
\end{problem}

The motivation for our study  comes from the analogous infinitesimal or tangential problems on $1$-cycles. The case of $0$-cycles is certainly simpler, as it is purely algebraic. Our approach for studying these problems in the $0$-cycle case is geometric, the main ingredients being the \emph{connection curve} $\Gamma_f$ or $\Gamma_{f+\epsilon g}$ and the \emph{zero hypersurface} $S_g$ or $S_f$ and their respective intersections. 

 In subsection \ref{reduction}, we showed how the tangential (or infinitesimal) problem on $1$-cycles reduces to a \emph{generalized} tangential or infinitesimal problem on $0$-cycles. Recall that the generalization consists in replacing a polynomial integrand $g$  by a multivalued function $G$ obtained by integration of a polynomial. 

\begin{problem}  
A natural challenging question is can our techniques for solving the $0$-dimensional tangential or infinitesimal problem be extended to the study of \emph{generalized} tangential of infinitesimal problems on $0$ cycles and thus give a solution of the original tangential or infinitesimal problems on $1$-cycles. 
\end{problem}


\begin{thebibliography}{11}


\bibitem{A} 
Arnold, V. I. {\em Arnold's problems.}
Translated and revised edition of the 2000 Russian original. With a preface by V. Philippov, A. Yakivchik and M. Peters. Springer-Verlag, Berlin; PHASIS, Moscow, 2004. xvi+639 pp.
\bibitem{ABCM} \'Alvarez A, Bravo JL, Christopher C , Marde\v si\'c P (2021) {\em Infinitesimal Center Problem on Zero Cycles and the Composition Conjecture}, Functional Analysis and Its Applications {\bf 55}, (4) 257-271.
\bibitem{ABM} A. \'Alvarez, J.L. Bravo, P. Marde\v si\'c, {\em Vanishing Abelian integrals on zero-dimensional cycles}, Proc. Lond. Math. Soc. (3) 107 (2013), no. 6, 1302–1330.
\bibitem{BNY}  Binyamini, Gal; Novikov, Dmitry; Yakovenko, Sergei, {\em On the number of zeros of Abelian integrals}. Invent. Math. 181 (2010), no. 2, 227–289.
\bibitem{CDR} M. Caubergh et al.F. Dumortier, R. Roussarie, C. R. {\em Alien limit cycles near a Hamiltonian 2-saddle cycle} Acad. Sci. Paris, Ser. I 340 (2005).
\bibitem{CM} C. Christopher, P. Marde\v{s}i\'{c}, {\em The monodromy problem and the tangential center problem}, Funkts. Anal. Prilozh., 44-1 (2010) 27--43.

\bibitem{F} Fran\c coise, J. P. \emph{Successive derivatives of a first return map, application to the study of quadratic vector fields},
Ergodic Theory Dynam. Systems 16 (1996), no. 1, 87–96.
\bibitem{G} L. Gavrilov {\em The inﬁnitesimal 16th Hilbert problem in the quadratic case} Invent. math. 143, 449–497 (2001)
\bibitem{GM} Gavrilov, Movasati, {\em The inﬁnitesimal 16th Hilbert problem in dimension zero}, Bull. Sci. math. 131 (2007) 242–257.
\bibitem{Giter} Gavrilov, L. \emph{Higher order Poincaré-Pontryagin functions and iterated path integrals}
Ann. Fac. Sci. Toulouse Math. (6) 14 (2005), no. 4, 663–682.
\bibitem{GP} L. Gavrilov, F. Pakovich, {\em Moments on Riemann surfaces and hyperelliptic Abelian integrals},
Comment. Math. Helv. 89 (2014), 125–155 DOI10.4171/CMH/314.
\bibitem{Ing} A. W. Ingleton {\em The Rank of Circulant Matrices},  J. London Math. Soc. s1-31 (4)  (1956): 445–460. 
\bibitem{Lang} S. Lang, {\em  Algebra}. Springer, New York, NY  (2002).
\bibitem{M} P. Marde\v si\'c,  {\em Chebyshev systems and the versal unfolding of the cusps of order  $n$}, 
Travaux en Cours, 57,
Hermann, Paris, 1998, xiv+153 pp.
\bibitem{Rou} R. Roussarie,  {\em Bifurcation of planar vector fields and Hilbert's sixteenth problem}
Progr. Math., 164
Birkh\"auser Verlag, Basel, 1998, xviii+204 pp.
\bibitem{RZ} C. Rousseau, H. Zoladek, {\em Zeros of complete elliptic integrals for $1:2$ resonance}, J. Diff. Eq. {\bf 94}, No.1, (1991), 41-54. 
\bibitem{ZL} Zhang Z and Li C 1993 {\em On the number of limit cycles of a class of quadratic Hamiltonian systems under quadratic
perturbations} Res. Rep. 33, 1997 Adv. Math. 26 445–60.


\end{thebibliography}
\end{document}